\documentclass[a4paper,11pt,twoside]{amsart}

\usepackage{graphicx}
\usepackage{amsfonts}
\usepackage{amsmath}
\usepackage{amsthm}
\usepackage{amssymb}
\usepackage[all]{xy}
\usepackage{hyperref}

\newtheorem{thm}{Theorem}[section]
\newtheorem{prop}[thm]{Proposition}
\newtheorem{lem}[thm]{Lemma}
\newtheorem{cor}[thm]{Corollary}

\numberwithin{equation}{section}

\theoremstyle{definition}
\newtheorem{definition}[thm]{Definition}
\newtheorem{remark}[thm]{Remark}
\newtheorem{ex}[thm]{Example}

\DeclareMathOperator{\im}{im} 
\DeclareMathOperator{\cok}{coker} 
\DeclareMathOperator{\spec}{Spec} 
\newcommand{\iso}{\cong}

\newcommand{\id}{\mathrm{id}}
\newcommand{\dual}{^{\vee}} 
\newcommand{\comp}{\circ} 
\newcommand{\kcomp}{\star} 
\newcommand{\mor}[1]{\xrightarrow{#1}}
\newcommand{\mono}{\hookrightarrow} 
\newcommand{\epi}{\twoheadrightarrow} 
\newcommand{\isomor}{\mor{\sim}} 
\newcommand{\rest}[1]{|_{#1}} 
\newcommand{\cp}[1]{#1^{\bullet}} 
\newcommand{\K}{\Bbbk} 
\newcommand{\cat}[1]{{\mathbf{#1}}} 
\newcommand{\opp}{^{\circ}} 
\newcommand{\Hom}{\mathrm{Hom}}
\newcommand{\Ext}{\mathrm{Ext}}

\newcommand{\tot}{\mathrm{tot}} 
\newcommand{\dg}{\mathrm{dg}}
\newcommand{\FM}[2][]{\Phi^{#1}_{#2}} 
\newcommand{\FMdg}[1]{\FM[\dg]{#1}} 
\newcommand{\FMS}[1]{\FM[\mathrm{s}]{#1}}
\newcommand{\D}[1][]{\mathrm{D}^{#1}} 
\newcommand{\Db}{\D[b]} 
\newcommand{\Dp}[1][]{\cat{Perf}_{#1}} 
\newcommand{\Dg}{\D[\dg]} 
\newcommand{\dgD}{\mathcal{D}} 
\newcommand{\Perf}{\mathrm{Perf}^{\,\dg}} 
\newcommand{\SF}{\mathrm{SF}} 
\newcommand{\SFfg}{\mathrm{SF_{fg}}} 
\newcommand{\rep}{\mathrm{rep}} 
\newcommand{\Ac}{\ka c} 
\newcommand{\Cdg}{\mathrm{C}^{\dg}} 
\newcommand{\Acdg}{\mathrm{Ac}^{\dg}} 
\newcommand{\rd}{\mathbf{R}} 
\newcommand{\ld}{\mathbf{L}} 
\newcommand{\lotimes}{\overset{\ld}{\otimes}} 
\newcommand{\fun}[1]{\mathsf{#1}} 
\newcommand{\dgfun}[1]{\mathsf{#1}^{\dg}} 
\newcommand{\Mod}[1]{\mathrm{Mod}\text{-}#1} 
\newcommand{\dgMod}[1]{\cal{M}od\text{-}#1} 
\newcommand{\Coh}{\cat{Coh}}
\newcommand{\Qcoh}{\cat{Qcoh}}
\newcommand{\p}{\mathrm{p}} 
\newcommand{\lto}{\longrightarrow}
\newcommand{\Hqe}{\mathbf{Hqe}}
\newcommand{\dgYon}[1][]{\fun{Y}^{#1}} 
\newcommand{\dYon}[1][]{\fun{Y}^{#1}} 
\newcommand{\newrho}{\fun{G}}

\newcommand{\Cc}{(\ast)}
\newcommand{\Ca}{(\vartriangle)}
\newcommand{\Cb}{(\diamond)}
\newcommand{\Ob}{\mathrm{Ob}}
\newcommand{\cal}{\mathcal}
\newcommand{\ka}{{\cal A}}
\newcommand{\kb}{{\cal B}}
\newcommand{\kc}{{\cal C}}
\newcommand{\kd}{{\cal D}}
\newcommand{\ke}{{\cal E}}
\newcommand{\kf}{{\cal F}}
\newcommand{\kg}{{\cal G}}

\newcommand{\kk}{{\cal K}}
\newcommand{\ki}{{\cal I}}

\newcommand{\kl}{{\cal L}}

\newcommand{\km}{{\cal M}}
\newcommand{\ko}{{\cal O}}

\newcommand{\NN}{\mathbb{N}}
\newcommand{\ZZ}{\mathbb{Z}}

\newcommand{\RR}{\mathbb{R}}
\newcommand{\CC}{\mathbb{C}}

\newcommand{\PP}{\mathbb{P}}

\newcommand{\Ind}{\fun{Ind}}
\newcommand{\Res}{\fun{Res}}

\begin{document}

	\title[Fourier--Mukai functors in the supported case]{Fourier--Mukai functors in the supported case}

	\author{Alberto Canonaco and Paolo Stellari}

	\address{A.C.: Dipartimento di Matematica ``F. Casorati'', Universit{\`a}
	degli Studi di Pavia, Via Ferrata 1, 27100 Pavia, Italy}
	\email{alberto.canonaco@unipv.it}

	\address{P.S.: Dipartimento di Matematica ``F.
	Enriques'', Universit{\`a} degli Studi di Milano, Via Cesare Saldini
	50, 20133 Milano, Italy}
	\email{paolo.stellari@unimi.it}
 \urladdr{http://users.unimi.it/stellari}

\thanks{A.~C.~ is partially supported by the national research project
  ``Moduli, strutture geometriche e loro applicazioni'' (PRIN 2009).
P.~S.~ is partially supported by the grants FIRB 2012 ``Moduli Spaces and Their Applications'' and
the national research project ``Geometria delle Variet\`a Proiettive'' (PRIN 2010-11).}


	
		\begin{abstract}
                  We prove that exact functors between the categories
                  of perfect complexes supported on projective schemes
                  are of Fourier--Mukai type if the functor satisfies
                  a condition weaker than being fully faithful. We
                  also get generalizations of the results in the
                  literature in the case without support conditions. Some
                  applications are discussed and, along the way, we
                  prove that the category of perfect supported complexes has a
                  strongly unique enhancement.
		\end{abstract}

		\maketitle

	\section{Introduction}\label{Intro}

One of the most intriguing open questions in the theory of derived categories is whether all exact functors between the categories of perfect complexes (or between the bounded derived categories of coherent sheaves) on projective schemes are of Fourier--Mukai type.
It might be worth recalling that, if $X_1$ and $X_2$ are projective schemes, an exact functor
$\fun{F}\colon\Dp(X_1)\to\Dp(X_2)$ between the corresponding categories of
perfect complexes is a \emph{Fourier--Mukai functor} (or of
\emph{Fourier--Mukai type}) if there exists $\ke\in\Db(X_1\times X_2)$
and an isomorphism of exact functors $\fun{F}\iso\FM{\ke}$. Here
$\FM{\ke}\colon\Dp(X_1)\to\Dp(X_2)$ is the exact functor defined by
\[
\FM{\ke}:=\rd(p_2)_*(\ke\lotimes p_1^*(-)),
\]
where $p_i\colon X_1\times X_2\to X_i$ is the natural projection. The complex $\ke$ is called a \emph{kernel} of $\fun{F}$.

While, in general, the kernel is certainly not unique (up to
isomorphism) due to \cite{CS1}, the question about the existence of
such kernels is widely open. Indeed, despite the fact that a
conjecture in \cite{BLL} would suggest a positive answer to it, for
the time being, only partial results in this direction are
available. Let us recall some of them. In \cite{Or} (together with
\cite{BB}) the case of exact fully faithful functors between the
bounded derived categories of coherent sheaves on smooth projective
varieties is completely solved by Orlov. Various generalizations to
quotient stacks and twisted categories were given in \cite{Ka} by
Kawamata and in \cite{CS} respectively. In particular, the main result
of \cite{CS} shows that all exact functors
$\fun{F}\colon\Db(X_1)\to\Db(X_2)$ such that
	\begin{equation}\label{eqn:cCS}
		\Hom_{\Db(X_2)}(\fun{F}(\ka),\fun{F}(\kb)[k])=0,
	\end{equation}
for any $\ka,\kb\in\Coh(X_1)$ and any integer $k<0$, are Fourier--Mukai functors and their kernels are unique, up to isomorphism.

\medskip

The inspiration for our results in this paper comes from the new
approach to the representability problem in \cite{LO}, where the
authors show that all exact fully faithful functors
$\fun{F}\colon\Dp(X_1)\to\Dp(X_2)$ between the categories of perfect
complexes on the projective schemes $X_1$ and $X_2$ are of
Fourier--Mukai type. To show this, Lunts and Orlov prove that such
fully faithful functors admit dg lifts. At that point, they can invoke
the representability result in \cite{To}. Indeed, To\"en proved that,
in the dg setting, all morphisms in the localization of the category of dg categories by quasi-equivalences are of Fourier--Mukai
type (in an appropriate dg sense). Notice that the strategy in \cite{LO} allows the authors to
improve the results in \cite{Ba1}.

To make clear the categorical setting we are going to work with, let
$X_1$ be a quasi-projective scheme containing a projective subscheme
$Z_1$ such that the structure sheaf $\ko_{iZ_1}$ of the $i$-th
infinitesimal neighbourhood of $Z_1$ in $X_1$ is in $\Dp(X_1)$,
for every $i>0$. This last condition is verified for instance when
either $Z_1=X_1$ or $X_1$ is smooth. Moreover let $X_2$ be a separated
scheme of finite type over the base field $\K$ with a closed subscheme $Z_2$. One can then consider the categories $\Dp[Z_i](X_i)$ of perfect complexes on $X_i$ with cohomology sheaves supported on $Z_i$. The definition of Fourier--Mukai functor makes perfect sense also in this context (see Definition \ref{def:FMs}).

\medskip

A rewriting of \eqref{eqn:cCS} in the supported setting which weakens the fully-faithfulness condition in \cite{LO,Or} requires a bit of care. Indeed, assuming $X_1$, $X_2$, $Z_1$ and $Z_2$ to be as above, one can consider exact functors $\fun{F}\colon\Dp[Z_1](X_1)\to\Dp[Z_2](X_2)$ such that

\medskip

\parbox{17pt}{$\Cc$}
\parbox{436pt}{\begin{itemize}
\item[(1)] \emph{$\Hom(\fun{F}(\ka),\fun{F}(\kb)[k])=0$, for
any $\ka,\kb\in\Coh_{Z_1}(X_1)\cap\Dp[Z_1](X_1)$ and any
integer $k<0$;
\item[(2)] For all $\ka\in\Dp[Z_1](X_1)$ with trivial
cohomologies in positive degrees, there is
$N\in\ZZ$ such that
\[
\Hom(\fun{F}(\ka),\fun{F}(\ko_{|i|Z_1}(jH_1)))=0,
\]
for any $i<N$ and any $j\ll i$, where $H_1$ is an ample (Cartier)
divisor on $X_1$.}
\end{itemize}}

\medskip

At first sight this condition may look a bit involved, but if
$Z_1=X_1$ is smooth with $\dim(X_1)>0$, then part $(2)$ of $\Cc$ is redundant and thus
$\Cc$ turns out to be equivalent to \eqref{eqn:cCS} (see Proposition
\ref{prop:smooth}).  In general full functors always satisfy $\Cc$, if we assume further that the maximal $0$-dimensional torsion subsheaf $T_0(\ko_{Z_1})$ of
$\ko_{Z_1}$ is trivial. Actually, due to \cite{COS}, a non-trivial full functor is automatically faithful if
$Z_1$ is connected. We will discuss in Section \ref{subsec:smoothex}
the existence of non-full functors with property $\Cc$.

We are now ready to state our first main result.

\begin{thm}\label{thm:main2}
Let $X_1$ be a quasi-projective scheme containing a projective
subscheme $Z_1$ such that $\ko_{iZ_1}\in\Dp(X_1)$, for all $i>0$, and let
$X_2$ be a separated scheme of finite type over the base field $\K$ with a
closed subscheme $Z_2$. Let
\[
\fun{F}\colon\Dp[Z_1](X_1)\lto\Dp[Z_2](X_2)
\]
be an exact functor.

If $\fun{F}$ satisfies $\Cc$, then there exist
$\ke\in\Db_{Z_1\times Z_2}(\Qcoh(X_1\times X_2))$ and an isomorphism
of functors $\fun{F}\iso\FMS{\ke}$. Moreover, if $X_i$ is smooth
quasi-projective, for $i=1,2$, and $\K$ is perfect, then $\ke$ is
unique up to isomorphism.
\end{thm}

This result summarizes the content of Proposition \ref{prop:rep1},
Lemma \ref{lem:boundedcohom} and Proposition \ref{prop:uniq}. The
proof is contained in Sections \ref{sec:existence} and
\ref{sect:uniqueness} and it uses the approach via dg categories
proposed in \cite{LO}. Clearly, assuming $X_i=Z_i$ for $i=1,2$, our result extends the one in \cite{LO}
about singular projective schemes (see Corollary \ref{cor:experf}). Notice that the symbol $\FMS{\ke}$
stands for the `supported' Fourier--Mukai functor with kernel $\ke$
defined precisely in \eqref{eqn:FMdef}.

\smallskip

Our second main result concerns the uniqueness of the enhancement for
the category of perfect supported complexes mentioned above and it is
proved in Section \ref{subsec:enh}.

\begin{thm}\label{thm:main1}
Let $X$ be a quasi-projective scheme containing a projective
subscheme $Z$ such that $\ko_{iZ}\in\Dp(X)$, for all $i>0$, and
$T_0(\ko_{Z})=0$. Then $\Dp[Z](X)$ has a strongly unique enhancement.
\end{thm}

The notion of enhancement and its strong uniqueness is discussed in
Section \ref{subsec:dg}. For the moment we can roughly think of an
enhancement of $\Dp[Z](X)$ as a (pretriangulated) dg category whose
homotopy category is equivalent to $\Dp[Z](X)$. The enhancement is
strongly unique if two such are (quasi-)equivalent at the dg category
level and such an equivalence satisfies some additional condition. It
is worth noticing that the particular case $X=Z$ is one of the main
results in \cite{LO} (see Corollary \ref{cor:LO1}).

\medskip

\noindent{\bf Motivations.} Due to the technical nature of Theorems \ref{thm:main2} and \ref{thm:main1}, some geometric motivations are certainly in order here. From our point of view the reason for studying exact functors between categories with support conditions is two-fold. On one side the conjecture in \cite{BLL} concerning the fact that all `geometric' functors are of Fourier--Mukai type appears extremely difficult to prove in complete generality. Thus it makes sense to test its validity weakening the assumptions on the geometric nature of the triangulated categories involved and on the exact functors between them. In this sense, this paper is in the same spirit as \cite{CS} and \cite{LO}.

\smallskip

On the other hand, one would like to study easy-to-handle
\emph{$d$-Calabi--Yau categories}, i.e.\ triangulated categories whose
Serre functor is isomorphic to the shift by the positive integer
$d$. Challenging examples are certainly provided by the
derived categories of smooth projective Calabi--Yau
threefolds. Indeed, the homological version of the Mirror Symmetry
conjecture \cite{Ko} for those threefolds involves these categories with implications for the
manifolds parametrizing stability conditions \cite{Br1} into which (up
to the quotient by the group of autoequivalences) the K\"{a}hler
moduli spaces embed. One big open problem in this direction is the
lack of examples of stability conditions for Calabi--Yau threefolds.

The group of autoequivalences of the derived category, besides being an interesting algebraic object in itself, acts on the stability manifold. Already for Calabi--Yau manifolds of dimension $2$ (i.e.\ K3 surfaces), this group is very complicated and one of the main motivations of \cite{Or} is to stimulate its study. As for stability conditions, in higher dimension the situation becomes much more involved.

Therefore, following suggestions from the physics literature, one may
start from the non-compact or the so called `open' Calabi--Yau's. Let
us be more precise discussing some explicit examples where the ambient
space $X_1$ is smooth and Theorem \ref{thm:main2} (or a variant of it) applies.

\smallskip

Following \cite{FY} and \cite{KYZ}, one can consider the triangulated
category $\cat{T}_S$ classically generated by a $d$-spherical object $S$ (here $d$
is a positive integer) in an idempotent complete triangulated category
$\cat{T}$. An object $S$ is $d$-spherical if the graded
algebra $\Ext^*(S,S)$ is isomorphic to the cohomology of a
$d$-sphere. We will study this example in Section
\ref{subsec:spherical} when $d=1$ as in this case $\cat{T}_S$ is
nothing but $\Db_\p(C)$, where $C$ is a smooth curve and $\p\in C$ is
a $\K$-rational point. Thus we obtain the following result, which is a
particular case of Proposition \ref{prop:spherical}.

\begin{prop}\label{prop:sph}
	Every exact autoequivalence of $\cat{T}_S$ is of Fourier--Mukai type if $S$ is a $1$-spherical object.
\end{prop}

This completes the picture in \cite{FY} which provides a description of the subgroup of Fourier--Mukai autoequivalences. We should remark here that the result above is not a direct consequence of Theorem \ref{thm:main2} as the maximal $0$-dimensional torsion subsheaf of $\ko_\p$ is obviously not trivial and part (2) of $\Cc$ does not hold true.

\smallskip

Interesting examples of $2$-Calabi--Yau categories are provided by the
local resolutions of $A_n$-singularities on surfaces which were
studied in \cite{IU,IUU}. More precisely, one considers
$Y=\spec(\CC[[x,y,z]]/(x^2+y^2+z^{n+1}))$ (the $A_n$-singularity), the
minimal resolution $f\colon X\to Y$ and $Z:=f^{-1}(p)$, where $p$ is the closed point in $Y$. Notice that, in this case, $T_0(\ko_{Z})=0$. The category one wants to consider is then $\Db_Z(X)=\Dp[Z](X)$ and using Theorem \ref{thm:main2} we can reprove in a direct way the following result already contained in \cite{IUU}.

\begin{cor}\label{cor:An}
Every exact autoequivalence of $\Db_Z(X)$ is of Fourier--Mukai type.
\end{cor}

\smallskip

Finally, to get examples of $3$-Calabi--Yau categories one can take the total space $\tot(\omega_{\PP^2})$ of the canonical bundle of $\PP^2$. In this case, if $Z$ denotes the zero section of the projection $\tot(\omega_{\PP^2})\to\PP^2$, the derived category $\Dp[Z](\tot(\omega_{\PP^2}))=\Db_Z(\tot(\omega_{\PP^2}))$ is a $3$-Calabi--Yau category and may be seen as an interesting example to test predictions about Mirror Symmetry and the topology of the space of stability conditions according to Bridgeland's definition (see \cite{BM} for results in this direction). Here again $T_0(\ko_{Z})=0$ and so Theorem \ref{thm:main2} yields the following.

\begin{cor}\label{cor:omega}
Every exact autoequivalence of $\Db_{Z}(\tot(\omega_{\PP^2}))$ is of Fourier--Mukai type.
\end{cor}

As an application of Proposition \ref{prop:spherical} and Theorem \ref{thm:main1}, the triangulated categories in the three examples above have strongly unique enhancements.

\medskip

\noindent{\bf The plan of the paper.} In Section \ref{sec:ample} we
provide the necessary preliminary material concerning derived
categories of supported sheaves, and we introduce the notion of an almost
ample set. Then we prove a criterion (generalizing others present in
the literature) for extending a morphism defined on a suitable subset
of the source category
between exact functors satisfying a condition related to $\Cc$. This is done in Section
\ref{sec:extending} using the notion of convolution. In Section
\ref{sec:existence} we deal with the existence of Fourier--Mukai
kernels and the strong uniqueness of enhancements. In particular, we need to generalize and to modify the argument in \cite{LO} to make it work in our setting. In the same section we also discuss the case of $1$-spherical objects. Section \ref{sect:uniqueness} deals with various questions about boundedness and uniqueness of Fourier--Mukai kernels.

\medskip

\noindent{\bf Notation.} In the paper, $\K$ is a field. All schemes
are assumed to be at least of finite type and separated over $\K$. All additive
(in particular, triangulated) categories and all additive (in
particular, exact) functors will be assumed to be $\K$-linear. If $\cat{A}$ is an abelian (or more generally an exact) category,
$\D(\cat{A})$ denotes the derived category of $\cat{A}$ and
$\Db(\cat{A})$ its full subcategory of complexes with bounded cohomology.
For an object
\[
A:=\{\cdots\to A^j\mor{d^j}A^{j+1}\mor{d^{j+1}}\cdots\mor{d^{i-1}}A^i\mor{d^i}A^{i+1}\to\cdots\}
\]
in $\D(\cat{A})$, we can consider the \emph{gentle truncations}
$\tau_{\leq i}A$, $\tau_{\geq i}A$, defined as
\begin{gather*}
\tau_{\leq i}A:=\{\cdots\to
A^j\mor{d^j}A^{j+1}\mor{d^{j+1}}\cdots\mor{d^{i-1}}\ker d^i\to0\to\cdots\}\\
\tau_{\geq i}A:=\{\cdots\to 0\to \cok d^{i-1}\to A^{i+1}\mor{d^{i+1}}\cdots\to A^j\mor{d^j}\cdots\}.
\end{gather*}
Unless clearly stated, all functors are derived even if, for simplicity, we
use the same symbol for a functor and its derived version. Natural
transformations (in particular, isomorphisms) between exact functors
are always assumed to be compatible with shifts.

\section{Preliminaries}\label{sec:ample}

The first part of this section provides a quick introduction to some basic and
well-known facts concerning the derived categories of supported
sheaves. Then we define and discuss the notion of an almost ample set.

\subsection{Categories with support conditions}\label{subsec:prelcat}

Let $X$ be a scheme and let $Z$ be
a closed subscheme of $X$. We denote by $\D_Z(\Qcoh(X))$ the derived
category of unbounded complexes of quasi-coherent sheaves on $X$ with
cohomologies supported on $Z$. We will be particularly interested in
the triangulated categories
\begin{equation}\label{eqn:categories}
\begin{split}
&\Db_Z(\Qcoh(X)):=\D_Z(\Qcoh(X))\cap\Db(\Qcoh(X))\\
&\Db_Z(X):=\D_Z(\Qcoh(X))\cap\Db(X),
\end{split}
\end{equation}
where $\Db(X):=\Db_\Coh(\Qcoh(X))$ is the full subcategory of
$\Db(\Qcoh(X))$ consisting of complexes with coherent cohomologies.
Denote by $\Dp(X)\subset\D(\Qcoh(X))$ the full subcategory of perfect
complexes on $X$.
Notice that $\Dp(X)\subseteq\Db(X)$ and, if $X$ is
quasi-projective, equality holds if and only if $X$ is regular. In the
supported case we set
\[
\Dp[Z](X):=\D_Z(\Qcoh(X))\cap\Dp(X).
\]
Thus, if $X$ is smooth, $\Dp[Z](X)=\Db_Z(X)$.

\begin{prop}\label{prop:catcomp} {\bf (\cite{R}, Theorems 5.3(i) and 6.8.)}
The category $\D_Z(\Qcoh(X))$ is compactly generated and the subcategory
of compact objects $\D_Z(\Qcoh(X))^c$ coincides with $\Dp[Z](X)$.
\end{prop}

Recall that an object $A$ in a triangulated category $\cat{T}$ is \emph{compact} if, given $\{X_i\}_{i\in I}\subset\cat{T}$ such
that $I$ is a set and $\bigoplus_i X_i$ exists in $\cat{T}$, the canonical map
\[
\bigoplus_i\Hom(A,X_i)\lto\Hom\left(A,\oplus_i X_i\right)
\]
is an isomorphism. Moreover, $\cat{T}$ is \emph{compactly generated}
if there is a set $S$ of objects in the full subcategory $\cat{T}^c$ of compact objects of $\cat{T}$ such that, given $E\in\cat{T}$ with $\Hom(A,E[i])=0$ for all $A\in S$ and all $i\in\ZZ$, then $E\iso 0$. For more details, the reader can consult \cite[Sect.\ 3.1]{R}.

The category $\D_Z(\Qcoh(X))$ is a full subcategory of $\D(\Qcoh(X))$ and let
\[
\iota\colon\D_Z(\Qcoh(X))\mono\D(\Qcoh(X))
\]
be the inclusion. We use the same symbol to denote the inclusion functor for the other categories in \eqref{eqn:categories}. According to \cite[Sect.\ 3]{L}, the functor $\iota$ has a right adjoint
\[
\iota^!\colon\D(\Qcoh(X))\to\D_Z(\Qcoh(X)).
\]
Notice that the existence of $\iota^!$ could be also deduced from \cite{N2}, since $\iota$ clearly commutes with arbitrary direct sums. As $\iota$ is fully faithful, we have
$\iota^!\comp\iota\iso\id$.

\begin{remark}\label{rmk:Lipman}
Actually \cite{L} deals only with modules over a commutative
noetherian ring. On the other hand, as it is observed in the
introduction of \cite{L}, most of its results globalize to sheaves
over schemes (in particular those used in this section).
\end{remark}

\begin{lem}\label{lem:sums}
The functor $\iota^!$ sends bounded complexes to bounded complexes and
commutes with direct sums.
\end{lem}

\begin{proof}
The fact that $\iota^!$ sends bounded complexes to bounded complexes
follows from \cite[Cor.\ 3.1.4]{L}. The commutativity with direct sums
is due to \cite[Cor.\ 3.5.2]{L}.
\end{proof}

\medskip

Now, let $X_1$ and $X_2$ be schemes containing, respectively, two closed
subschemes $Z_1$ and $Z_2$. We will denote
by $\iota_i$ (respectively $\iota_{i,j}$) the inclusion morphisms
relative to the pair $(X_i,Z_i)$ (respectively $(X_i\times
X_j,Z_i\times Z_j)$), for $i,j=1,2$.

\begin{lem}\label{lem:compat}
Let $f\colon X_1\to X_2$ be a morphism of schemes such that
$f^{-1}(Z_2)=Z_1$. Then there are isomorphisms of exact functors
\begin{gather*}
f_*\comp\iota_1\comp\iota_1^!\iso\iota_2\comp\iota_2^!\comp f_*\colon
\D(\Qcoh(X_1))\to\D(\Qcoh(X_2)) \\
f^*\comp\iota_2\comp\iota_2^!\iso\iota_1\comp\iota_1^!\comp f^*\colon
\D(\Qcoh(X_2))\to\D(\Qcoh(X_1)).
\end{gather*}
Moreover, given a pair $(X,Z)$, for every $\kf\in\D(\Qcoh(X))$ there is
an isomorphism of exact functors
\[
\iota\comp\iota^!(\kf\otimes-)\iso\kf\otimes\iota\comp\iota^!(-)\colon
\D(\Qcoh(X))\to\D(\Qcoh(X)).
\]
\end{lem}

\begin{proof}
See Cor.\ 3.4.3, 3.4.4 and 3.3.1 of \cite{L}.
\end{proof}

\begin{definition}\label{def:FMs}
An exact functor
\[
\fun{F}\colon\D_{Z_1}(\Qcoh(X_1))\to\D_{Z_2}(\Qcoh(X_2))
\]
is a \emph{Fourier--Mukai functor} if there exists $\ke\in\D_{Z_1\times Z_2}(\Qcoh(X_1\times X_2))$ and an isomorphism of exact functors
\begin{equation}\label{eqn:FMdef}
\fun{F}\iso\FMS{\ke}:=\iota_2^!\comp(p_2)_*(\iota_{1,2}(\ke)\otimes p_1^*\comp\iota_1(-))
\end{equation}
where $p_i\colon X_1\times X_2\to X_i$ is the projection.
\end{definition}

Analogous definitions can be given for functors defined between bounded derived categories of quasi-coherent, coherent or perfect complexes. The object $\ke$ is called \emph{Fourier--Mukai kernel}. We will use the standard notation $\FM{\ke}$ when $Z_i= X_i$ or to denote Fourier--Mukai functors between $\D(\Qcoh(X_1))$ and $\D(\Qcoh(X_2))$.

Observe that, as $(p_2)_*(\iota_{1,2}(\ke)\otimes p_1^*\comp\iota_1(-))$
is supported on $Z_2$, $\iota_2\comp\FMS{\ke}\iso\FM{\iota_{1,2}(\ke)}
\comp\iota_1$. This clearly implies that
$\FMS{\ke}\iso\iota_2^!\comp\FM{\iota_{1,2}(\ke)}\comp\iota_1$, and
more generally we can prove the following result.

\begin{lem}\label{lem:FMS}
Under the above assumptions, for every $\widetilde{\ke}\in
\D(\Qcoh(X_1\times X_2))$ such that
$\ke\iso\iota_{1,2}^!(\widetilde{\ke})$ there exists an isomorphism of
exact functors
\[
\FMS{\ke}\iso\iota_2^!\comp\FM{\widetilde{\ke}}\comp\iota_1\colon
\D_{Z_1}(\Qcoh(X_1))\to\D_{Z_2}(\Qcoh(X_2)).
\]
\end{lem}

\begin{proof}
Denoting by $\bar{\iota}_1$ and $\bar{\iota}_2$ the inclusion
morphisms relative to the pairs $(X_1\times X_2,Z_1\times X_2)$ and
$(X_1\times X_2,X_1\times Z_2)$ respectively, it is easy to see that
$\iota_{1,2}\comp\iota_{1,2}^!\iso
\bar{\iota}_2\comp\bar{\iota}_2^!\comp\bar{\iota}_1\comp\bar{\iota}_1^!$.
Therefore, using Lemma \ref{lem:compat} repeatedly, we obtain
\begin{multline*}
\FMS{\ke}\iso\iota_2^!\comp(p_2)_*(\iota_{1,2}\comp\iota_{1,2}^!
(\widetilde{\ke})\otimes p_1^*\comp\iota_1(-))\iso
\iota_2^!\comp(p_2)_*(\bar{\iota}_2\comp\bar{\iota}_2^!\comp
\bar{\iota}_1\comp\bar{\iota}_1^!
(\widetilde{\ke})\otimes p_1^*\comp\iota_1(-)) \\
\iso\iota_2^!\comp(p_2)_*\comp\bar{\iota}_2\comp\bar{\iota}_2^!
(\bar{\iota}_1\comp\bar{\iota}_1^!
(\widetilde{\ke})\otimes p_1^*\comp\iota_1(-))\iso
\iota_2^!\comp\iota_2\comp\iota_2^!\comp(p_2)_*
(\bar{\iota}_1\comp\bar{\iota}_1^!
(\widetilde{\ke})\otimes p_1^*\comp\iota_1(-)),
\end{multline*}
whence
\begin{multline*}
\FMS{\ke}\comp\iota_1^!\iso
\iota_2^!\comp(p_2)_*(\bar{\iota}_1\comp\bar{\iota}_1^!
(\widetilde{\ke})\otimes p_1^*\comp\iota_1\comp\iota_1^!(-))\iso
\iota_2^!\comp(p_2)_*(\bar{\iota}_1\comp\bar{\iota}_1^!
(\widetilde{\ke})\otimes
\bar{\iota}_1\comp\bar{\iota}_1^!\comp p_1^*(-)) \\
\iso\iota_2^!\comp(p_2)_*(\widetilde{\ke}\otimes
\bar{\iota}_1\comp\bar{\iota}_1^!\comp p_1^*(-))\iso
\iota_2^!\comp(p_2)_*(\widetilde{\ke}\otimes
p_1^*\comp\iota_1\comp\iota_1^!(-))=
\iota_2^!\comp\FM{\widetilde{\ke}}\comp\iota_1\comp\iota_1^!.
\end{multline*}
So we conclude that
$\FMS{\ke}\iso\FMS{\ke}\comp\iota_1^!\comp\iota_1\iso
\iota_2^!\comp\FM{\widetilde{\ke}}\comp\iota_1\comp\iota_1^!\comp\iota_1
\iso\iota_2^!\comp\FM{\widetilde{\ke}}\comp\iota_1$.
\end{proof}

\medskip

Consider the abelian categories $\Qcoh_Z(X)$ and $\Coh_Z(X)$
consisting of quasi-coherent and, respectively, coherent sheaves
supported on $Z$.
The following will be implicitly used at many points of this paper.

\begin{prop}\label{prop:Ball2} {\bf (\cite{Ba}, Lemma 3.3, Corollary 3.4.)}
The natural functors $\D(\Qcoh_Z(X))\to\D_Z(\Qcoh(X))$ and $\Db(\Coh_Z(X))\to\Db_Z(X)$ are equivalences.
\end{prop}

Notice that the proof in \cite{Ba} works in our generality as well.
Denoting by $i\colon Z\mono X$ the closed embedding, we will also
need the following result.

\begin{prop}\label{prop:Ball1} {\bf (\cite{Ba}, Lemma 3.6, Corollary 3.7.)}
{\rm (i)} The functor $i_*$ is exact and the image of $i_*$ generates $\Coh_Z(X)$ as an abelian category, i.e.\ the smallest abelian subcategory of $\Coh_Z(X)$ closed under extensions and containing the essential image of $i_*$ is $\Coh_Z(X)$ itself. Similarly, the  smallest abelian subcategory closed under extensions and arbitrary direct sums containing the image of $i_*$ in $\Qcoh_Z(X)$ is $\Qcoh_Z(X)$ itself.

{\rm (ii)} The image of $\Db(Z)$ under $i_*$ classically generates $\Db_Z(X)$
and the image under $i_*$ of $\D(\Qcoh(Z))$ classically completely generates $\D_Z(\Qcoh(X))$.
\end{prop}

Recall that, according to \cite{R}, a subcategory $\cat{S}$ of a triangulated category $\cat{T}$ \emph{classically generates} $\cat{T}$ if
the smallest thick triangulated subcategory of $\cat{T}$ containing $\cat{S}$ is $\cat{T}$ itself. On the other hand, $\cat{S}$ \emph{classically completely generates} $\cat{T}$ if $\cat{T}$ is the smallest thick subcategory which is closed under direct sums and contains $\cat{S}$.

\subsection{Almost ample sets}\label{subsec:almostample}

The attempt of this section is to provide a generalization of the notion of weakly ample sequence in \cite[Appendix A.2]{IUU} which, in turn, is a generalization of the usual definition of ample sequence (see, for example, \cite{Or}).

\begin{definition}\label{def:almostample}
Given an abelian category $\cat{A}$ and a set $I$, a subset
$\{P_i\}_{i\in I}\subseteq\cat{A}$ is an \emph{almost ample set}
if, for any $A\in\cat{A}$, there exists $i\in I$ such
that:
\begin{enumerate}
\item\label{ample2} there is a natural number $k$ and an epimorphism
$P_i^{\oplus k}\epi A$;
\item\label{ample3} $\Hom_{\cat{A}}(A,P_i)=0$.
\end{enumerate}
\end{definition}

\begin{remark}\label{rmk:ampleset}
Every ample sequence and, more generally, every weakly ample sequence is an almost
ample set (with $I=\ZZ$). It is also obvious by definition that every set containing
an almost ample set (respectively a set of objects satisfying \eqref{ample2} in Definition \ref{def:almostample}) is almost ample (respectively it satisfies \eqref{ample2} in Definition \ref{def:almostample}), too.
\end{remark}

To provide examples of almost ample sets which are suited for the
supported setting we are working in, let $X$ be a quasi-projective
scheme and let $Z$ be a projective subscheme of $X$. Assume further
that $\ko_{iZ}\in\Dp(X)$, for all $i>0$. Take $H$ an ample divisor on $X$ and define the subset of $\Coh_Z(X)$
\begin{equation}\label{eqn:almamp}
\cat{Amp}(Z,X,H):=\{\ko_{|i|Z}(jH))\}_{(i,j)\in\ZZ\times\ZZ}.
\end{equation}
When needed, we will think of $\cat{Amp}(Z,X,H)$ as the corresponding full subcategory of $\Coh_Z(X)$.

\begin{ex}\label{ex:geosetting}
	There are two interesting geometric situations for which $\ko_{iZ}\in\Dp(X)$, for all $i>0$, and thus $\cat{Amp}(Z,X,H)$ is contained in $\Dp[Z](X)$. Namely one can take $X$ to be a quasi-projective scheme containing a projective subscheme $Z$ such that either $Z=X$ or $X$ is smooth.
\end{ex}

The following result will be essential for the rest of the
paper. Using, for example, the notation of \cite[Def.\ 1.1.4]{HL}, we
denote by $T_0(\ko_Z)$ the maximal subsheaf of $\ko_Z$ whose support
has dimension $0$. For short, we call $T_0(\ko_Z)$ the \emph{maximal
  $0$-dimensional torsion subsheaf of $\ko_Z$}.

\begin{prop}\label{prop:amp}
Assume that $X$, $Z$ and $H$ are as above. Then $\cat{Amp}(Z,X,H)$
satisfies \eqref{ample2} in Definition
\ref{def:almostample}, and provides a set of (compact) generators of
the Grothendieck category $\Qcoh_Z(X)$. Moreover, if $T_0(\ko_Z)=0$,
then $\cat{Amp}(Z,X,H)$ is an almost ample set in $\Coh_Z(X)$.

More precisely, for any $\ka\in\Coh_Z(X)$, there is $N\in\ZZ$ such
that any $\ko_{|i|Z}(jH)$ with $i<N$ and $j\ll i$ satisfies \eqref{ample2} (and \eqref{ample3} if $T_0(\ko_Z)=0$)
in Definition \ref{def:almostample}.
\end{prop}

\begin{proof}
Notice that under the assumption $\ko_{iZ}\in\Dp(X)$ for all $i>0$,
the objects $\ko_{|i|Z}(jH)$ are compact as well for all $i,j\in\ZZ$.

Let $\ke$ be a sheaf in $\Coh_Z(X)$. To prove property \eqref{ample2},
observe that by \cite[Lemma 7.40]{R} there is an integer $N\le0$ such
that $\ke$ is a coherent $\ko_{|i|Z}$-module for all $i<N$. As $|i|Z$
is a projective scheme, it follows that for $i<N$ and $j\ll i$ there
is an epimorphism $\ko_{|i|Z}(jH)^{\oplus k}\epi\ke$ (for some
$k\in\NN$) in $\Coh(|i|Z)$, whence also in $\Coh_Z(X)$.

Assuming $T_0(\ko_Z)=0$ (which, indeed, implies $T_0(\ko_{|i|Z})=0$),
property \eqref{ample3} is easily verified taking $i<N$ and $j\ll
i$. In fact, for $i<N$ and any $j\in\ZZ$, we have
\[
\Hom_{\Coh_Z(X)}(\ke,\ko_{|i|Z}(jH))\iso\Hom_{\Coh(|i|Z)}(\ke,\ko_{|i|Z}(jH)).
\]
Under the assumption $T_0(\ko_{|i|Z})=0$, the same argument as in the proof of \cite[Prop.\ 9.2]{LO} yields that, for $j$ sufficiently small, the latter vector space is trivial.

To prove that $\cat{Amp}(Z,X,H)$ is a set of generators for the
category $\Qcoh_Z(X)$, it is enough to observe that, in view of
\eqref{ample2} of Definition \ref{def:almostample} and the fact that
any quasi-coherent sheaf is the direct limit of its coherent
subsheaves, for any $\ke\in\Qcoh_Z(X)$ there is a surjection
$\bigoplus_{j\in S}P_j\epi\ke$ where $S$ is a set and
$P_j\in\cat{Amp}(Z,X,H)$ for all $j\in S$ (see \cite[Sect.\ 8.3]{KS}).
\end{proof}

\begin{ex}\label{ex:IUU}
If $X$ is the resolution of an $A_n$-singularity and $Z$ is the
exceptional locus, a special case of almost ample set for $\Coh_Z(X)$
is provided by the \emph{weak ample sequence} $\cat{C}$ in
\cite[Appendix A]{IUU}, where $\cat{C}=\{\ko_{|i|Z}(iH)\in\cat{Amp}(Z,X,H):i\in\ZZ\}$.
\end{ex}

\section{Extending natural transformations}\label{sec:extending}

In this section we deal with the second key ingredient in our proof,
namely a criterion to extend natural transformations (in particular,
isomorphisms) between functors. We tried to put this result in a
generality which goes beyond the scope of this paper but which may be
useful in future works (see, for example, \cite[Prop.\ 5.15]{CS3}).

\subsection{Convolutions}\label{subsec:convolutions}

In this section we collect some well-known facts about convolutions which will be used in the paper. Most of the terminology in taken from \cite{Ka,Or} (see also \cite{CS}).

A bounded complex in a
triangulated category $\cat{T}$ is a sequence of objects and
morphisms in $\cat{T}$
\begin{equation}\label{eqn:complex}
A_{m}\mor{d_{m}}A_{m-1}\mor{d_{m-1}}\cdots\mor{d_{1}}A_0
\end{equation}
such that $d_{j}\comp d_{j+1}=0$ for $0<j<m$. A \emph{right convolution} of
\eqref{eqn:complex} is an object $A$ together with a morphism
$d_0\colon A_0\to A$ such that there exists a diagram in $\cat{T}$
\[
\xymatrix{A_m \ar[rr]^-{d_m} \ar[dr]_-{\id} \ar@{}[drr]|{\circlearrowright}
& & A_{m-1} \ar[rr]^-{d_{m-1}} \ar[dr] \ar@{}[drr]|{\circlearrowright} & &
\cdots \ar[rr]^-{d_2} & & A_1 \ar[rr]^-{d_1} \ar[dr]
\ar@{}[drr]|{\circlearrowright} & & A_0 \ar[dr]_-{d_0} \\
 & A_m \ar[ru] & & C_{m-1} \ar[ll]^-{[1]} \ar[ru] & & \cdots
\ar[ll]^-{[1]} & & C_1 \ar[ll]^-{[1]} \ar[ru] & & A, \ar[ll]^-{[1]}}
\]
where the triangles with a $\circlearrowright$ are
commutative and the others are distinguished.

\smallskip

Let $d_0\colon A_0\to A$ be a right convolution of \eqref{eqn:complex}. If
$\cat{T}'$ is another triangulated category and
$\fun{G}\colon\cat{T}\to\cat{T}'$ is an exact functor, then
$\fun{G}(d_0)\colon\fun{G}(A_0)\to\fun{G}(A)$ is a right convolution of
\[
\fun{G}(A_{m})\mor{\fun{G}(d_{m})}\fun{G}(A_{m-1})\mor{\fun{G}(d_{m-1})}
\cdots\mor{\fun{G}(d_{1})}\fun{G}(A_0).
\]
The following results will be used in the rest of this section. Notice that left convolutions in the notation of \cite{Ka} correspond to right convolutions in the present paper (according to the notation in \cite{CS,Or}).

\begin{lem}\label{lem:conv1} {\bf (\cite{Ka}, Lemmas 2.1 and 2.4.)}
Let \eqref{eqn:complex} be a complex in $\cat{T}$ satisfying
\begin{equation}\label{eqn:condconv}
\Hom_{\cat{T}}(A_a,A_b[r])=0\mbox{ for any $a>b$ and $r<0$}.
\end{equation}
Then \eqref{eqn:complex} has a right convolution which is uniquely
determined up to isomorphism (in general non canonical).
\end{lem}

\begin{lem}\label{lem:conv2} {\bf (\cite{CS}, Lemma 3.3.)}
Let
\[
\xymatrix{A_m \ar[rr]^-{d_m} \ar[d]^{f_m} & & A_{m-1}
\ar[rr]^-{d_{m-1}} \ar[d]^{f_{m-1}} & & \cdots \ar[rr]^-{d_2} & & A_1
\ar[rr]^-{d_1} \ar[d]^{f_1} & & A_0 \ar[d]^{f_0}\\
B_m \ar[rr]^-{e_m} & & B_{m-1} \ar[rr]^-{e_{m-1}} & & \cdots
\ar[rr]^-{e_2} & & B_1 \ar[rr]^-{e_1} & & B_0}
\]
be a morphism of complexes both satisfying \eqref{eqn:condconv} and
such that
\[
\Hom_{\cat{T}}(A_a,B_b[r])=0\mbox{ for any $a>b$ and $r<0$}.
\]
Assume that the corresponding right convolutions are of the
form $(d_0,0)\colon A_0\to A\oplus\bar{A}$ and $(e_0,0)\colon B_0\to
B\oplus\bar{B}$ and that
$\Hom_{\cat{T}}(A_p,B[r])=0$ for $r<0$ and any $p$. Then there
exists a unique morphism $f\colon A\to B$
such that $f\comp d_0=e_0\comp f_0$. If moreover each $f_i$ is an isomorphism,
then $f$ is an isomorphism as well.
\end{lem}

Let $\cat{T}:=\Db(\cat{A})$ for some abelian category $\cat{A}$ and
let $E$ be a complex as in \eqref{eqn:complex} and such that every
$A_i$ is an object of $\cat{A}$. Then a right convolution of $E$
(which is unique up to isomorphism by Lemma \ref{lem:conv1}) is the
natural morphism $A_0\to\cp{E}$, where $\cp{E}$ is the object of
$\Db(\cat{A})$ naturally associated to $E$ (namely, $E^i:=A_{-i}$ for
$-m\le i\le0$ and otherwise $E^i:=0$, with differential $d_{-i}\colon E^i\to
E^{i+1}$ for $-m\le i<0$).

\subsection{The criterion: extension to a subcategory}\label{subsec:crit1}

Looking carefully at the proof of \cite[Prop.\ 3.7]{CS}, one sees that the notion of ample sequence can be replaced there by the one of almost ample set. In particular, if $\cat{T}$ is a
triangulated category and $\cat{A}$ is an abelian category, we can deal with exact functors $\fun{F}\colon\Db(\cat{A})\to\cat{T}$ satisfying the following condition:
\smallskip
\begin{itemize}
\item[$\Cb$] \emph{$\Hom_{\cat{T}}(\fun{F}(A),\fun{F}(B)[k])=0$,
for any $A,B\in\cat{A}$ and any $k<0$.}
\end{itemize}

\smallskip

Hence one can prove the following result.

\begin{prop}\label{prop:extending}
Let $\cat{T}$ be a triangulated category and let $\cat{A}$ be an abelian category of finite homological dimension. Assume
that $\{P_{i}\}_{i\in I}\subseteq\cat{A}$ is an almost ample set and
denote by $\cat{C}$ the corresponding full subcategory. Let
$\fun{F}_1,\fun{F}_2\colon\Db(\cat{A})\to\cat{T}$ be exact functors
and let $f\colon\fun{F}_1\rest{\cat{C}}\isomor\fun{F}_2\rest{\cat{C}}$ be
an isomorphism of functors. Assume moreover the following:
\begin{itemize}
\item[(i)] the functor $\fun{F}_1$ satisfies $\Cb$;
\item[(ii)] $\fun{F}_1$ has a left adjoint.
\end{itemize}
Then there exists an isomorphism of exact functors
$g\colon\fun{F}_1\isomor\fun{F}_2$ extending $f$.
\end{prop}

In the rest of this paper we would like to apply Proposition
\ref{prop:extending} but, unfortunately, in the supported case a
functor $\fun{F}\colon\Db_{Z_1}(X_1)\to\Db_{Z_2}(X_2)$ may not have left or
right adjoint. Thus we are going to prove a more general result
(Proposition \ref{prop:extending1}, whose proof is however much
inspired by those of \cite[Prop.\ 2.16]{Or} and \cite[Prop.\ 3.7]{CS}),
from which Proposition \ref{prop:extending} will follow easily (see
the end of Section \ref{subsec:crit2}). To this purpose, we first introduce the
categorical setting which will be used in the rest of Section
\ref{sec:extending}.

\smallskip

Indeed, to weaken condition $\Cb$, let $\cat{E}$ be a full exact
subcategory of an abelian category $\cat{A}$ satisfying
the following conditions:
\begin{itemize}
\item[(E1)] A morphism in
$\cat{E}$ is an admissible epimorphism if and only if it is an
epimorphism in $\cat{A}$;
\item[(E2)] There is a set
$\{P_{i}\}_{i\in I}\subseteq \cat{E}$ which satisfies property \eqref{ample2} of Definition \ref{def:almostample} as objects of $\cat{A}$;
\item[(E3)] For all
$A\in\Db(\cat{E})\cap\cat{A}$, there exists an integer $N(A)$ such that
$\Hom_{\Db(\cat{A})}(A,B[i])=0$, for every $i>N(A)$ and every
$B\in\Db(\cat{E})\cap\cat{A}$.
\end{itemize}
The reader who is not familiar with the language of exact categories
can have a look at \cite{K2} (where admissible epimorphisms are called
deflations).

\begin{remark}\label{rmk:exact}
Under conditions (E1) and (E2), $\Db(\cat{E})$ can be identified with a
full subcategory of $\Db(\cat{A})$ by \cite[Thm.\ 12.1]{K2} and \cite[Sect.\ 4.1]{Ke} (or rather
its dual version). Indeed, notice
that, by \eqref{ample2} of Definition \ref{def:almostample}, for each
object $A$ of $\cat{A}$ there is an object $E$ in $\cat{E}$ with an epimorphism $E\epi A$ in
$\cat{A}$.
\end{remark}

For $\cat{E}$ and $\cat{A}$ satisfying (E1) and (E2), we will consider
exact functors $\fun{F}\colon\Db(\cat{E})\to\cat{T}$ (for $\cat{T}$ a
triangulated category) such that

\smallskip

\parbox{17pt}{$\Ca$}
\parbox{436pt}{
\begin{itemize}
\item[(1)] \emph{$\Hom(\fun{F}(A),\fun{F}(B)[k])=0$, for any $A,B\in\Db(\cat{E})\cap\cat{A}$ and any integer $k<0$;
\item[(2)] For all $C\in\Db(\cat{E})$ with trivial cohomologies in
positive degrees, there is $i\in I$ such that
\[
\Hom(\fun{F}(C),\fun{F}(P_{i}))=0
\]
and $i$ satisfies property \eqref{ample2} of Definition
\ref{def:almostample} for $H^0(C)$.}
\end{itemize}}

\smallskip

In order to state our first extension result, we need some more
notation. Let $\cat{C}$ be the full subcategory of $\cat{A}$ with objects
$\{P_i\}_{i\in I}$ and set $\cat{D}_0$ to be the strictly full
subcategory of $\Db(\cat{E})$ whose objects are isomorphic to shifts
of objects of $\cat{A}$. Recall that a full subcategory is strictly full it is closed under isomorphisms.

\begin{prop}\label{prop:extab}
Let $\cat{T}$ be a triangulated
category and let $\cat{E}$ be a full exact subcategory of an abelian category $\cat{A}$ satisfying {\rm (E1)}, {\rm (E2)} and {\rm (E3)}.
Let $\fun{F}_1,\fun{F}_2\colon\Db(\cat{E})\to\cat{T}$ be
exact functors with a natural transformation
$f\colon\fun{F}_1\rest{\cat{C}}\to\fun{F}_2\rest{\cat{C}}$. Assume moreover the following:
\begin{itemize}
\item[(i)] $\fun{F}_1$ and $\fun{F}_2$ both satisfy condition $(1)$
of $\Ca$;
\item[(ii)] For any $A,B\in\Db(\cat{E})\cap\cat{A}$ and any integer $k<0$
\[
\Hom(\fun{F}_1(A),\fun{F}_2(B)[k])=0.
\]
\end{itemize}
Then there exists a unique natural transformation compatible
with shifts
$f_0\colon\fun{F}_1\rest{\cat{D}_0}\to\fun{F}_2\rest{\cat{D}_0}$
extending $f$.
\end{prop}

\begin{proof}
For any $i\in
I$, let $f_i:=f(P_i)\colon\fun{F}_1(P_i)\to \fun{F}_2(P_i)$. We also set
$\cat{F}:=\Db(\cat{E})\cap\cat{A}$.

The first key step consists in showing that $f$ extends uniquely to
a natural transformation
$\fun{F}_1\rest{\cat{F}}\to\fun{F}_2\rest{\cat{F}}$. To this purpose,
one starts with $A\in\cat{F}$ and takes a(n infinite) resolution
\begin{equation}\label{eqn:ampleres}
\cdots\to P_{i_j}^{\oplus k_j}\mor{d_j}P_{i_{j-1}}^{\oplus
k_{j-1}}\mor{d_{j-1}}\cdots\mor{d_1}P_{i_0}^{\oplus
k_0}\mor{d_0}A\to0,
\end{equation}
where $i_j\in I$ and $k_j\in\NN$ for every $j\in\NN$.
To get it, one argues as follows. By \eqref{ample2} of Definition
\ref{def:almostample}, there exist $P_{i_0}$, $k_0$ and $d_0$ as
in \eqref{eqn:ampleres}. As $K_0:=\ker d_0$ is again an object of
$\cat{A}$, we can apply the same argument to it, getting an
epimorphism $d_1\colon P_{i_1}^{\oplus k_1}\to K_0$. Then one proceeds
by induction.

Let $N(A)$ be as in (E3), fix $m>N(A)$ and consider the bounded complex
\[
R_m:=\{P_{i_m}^{\oplus k_m}\mor{d_m}P_{i_{m-1}}^{\oplus
k_{m-1}}\mor{d_{m-1}}\cdots\mor{d_1}P_{i_0}^{\oplus k_0}\}.
\]
We can think of $R_m$ as an object in $\Db(\cat{E})$ which is then a (unique up to isomorphism) convolution
of $R_m$ itself. Moreover, if we set $K_m:=\ker d_m\in\cat{A}$, the object $R_m$ in
$\Db(\cat{E})$ sits in a distinguished triangle
\[
K_m[m]\lto R_m\lto A\lto K_m[m+1]
\]
in $\Db(\cat{E})$, and so $K_m\in\cat{F}$. Observe that,
due to the choice of $m$, we have
\[
\Hom_{\Db(\cat{A})}(A,K_m[m+1])\iso\Hom_{\Db(\cat{A})}(A,K_m[m])\iso0,
\]
implying $R_m\iso K_m[m]\oplus A$. Hence, we conclude that a (unique up to isomorphism) convolution
of $R_m$ is $(d_0,0)\colon P_{i_0}^{\oplus k_0}\to A\oplus K_m[m]$,

Hence for $q\in\{1,2\}$ the complex
\[
\fun{F}_q(R_m):=\{\fun{F}_q(P_{i_m}^{\oplus k_m})
\mor{\fun{F}_q(d_m)}\fun{F}_q(P_{i_{m-1}}^{\oplus k_{m-1}})
\mor{\fun{F}_q(d_{m-1})}\cdots\mor{\fun{F}_q(d_1)}
\fun{F}_q(P_{i_0}^{\oplus k_0})\}
\]
admits a convolution $(\fun{F}_q(d_0),0)\colon\fun{F}_q(P_{i_0}^{\oplus k_0})\to
\fun{F}_q(A\oplus K_m[m])$. Lemma \ref{lem:conv1} and condition
(i) ensure that such a convolution is unique up to
isomorphism. Moreover, by (ii),
\[
\Hom_{\cat{T}}(\fun{F}_1(P_{i_j}),\fun{F}_2(P_{i_k})[r])\iso
\Hom_{\cat{T}}(\fun{F}_1(P_{i_l}),\fun{F}_2(A)[r])\iso0
\]
for any $i_j,i_k,i_l\in\{i_0,\ldots,i_m\}$ and $r<0$. Hence we can
apply Lemma \ref{lem:conv2} getting a unique morphism $f_A\colon
\fun{F}_1(A)\to \fun{F}_2(A)$ making the following diagram commutative:
\[
\xymatrix{\fun{F}_1(P_{i_m}^{\oplus k_m}) \ar[rr]^-{\fun{F}_1(d_m)}
\ar[d]^{f_{i_m}^{\oplus k_m}} & & \fun{F}_1(P_{i_{m-1}}^{\oplus k_{m-1}})
\ar[rr]^-{\fun{F}_1(d_{m-1})} \ar[d]^{f_{i_{m-1}}^{\oplus k_{m-1}}} & &
\cdots \ar[rr]^-{\fun{F}_1(d_1)} & & \fun{F}_1(P_{i_0}^{\oplus k_0})
\ar[rr]^-{\fun{F}_1(d_0)} \ar[d]^{f_{i_0}^{\oplus k_0}} & &
\fun{F}_1(A)\ar[d]^{f_A} \\
\fun{F}_2(P_{i_m}^{\oplus k_m}) \ar[rr]^-{\fun{F}_2(d_m)} & &
\fun{F}_2(P_{i_{m-1}}^{\oplus k_{m-1}}) \ar[rr]^-{\fun{F}_2(d_{m-1})}
& & \cdots\ar[rr]^-{\fun{F}_2(d_1)} & &
\fun{F}_2(P_{i_0}^{\oplus k_0}) \ar[rr]^-{\fun{F}_2(d_0)} & &
\fun{F}_2(A).}
\]
By Lemma \ref{lem:conv2}, the definition of $f_A$ does not depend on
the choice of $m$. In other words, if we choose a different $m'>N(A)$
and we truncate \eqref{eqn:ampleres} in position $m'$, the bounded
complexes $\fun{F}_q(R_{m'})$ give rise to the same morphism $f_A$.

To show that the definition of $f_A$ does not depend on the choice
of the resolution \eqref{eqn:ampleres}, consider another resolution of
$A$
\begin{equation}\label{eqn:reso2}
\cdots\to P_{i'_j}^{\oplus k'_j}\mor{d'_j}P_{i'_{j-1}}^{\oplus
k'_{j-1}}\mor{d'_{j-1}}\cdots\mor{d'_1}P_{i'_0}^{\oplus
k'_0}\mor{d'_0} A\to0.
\end{equation}
and denote by $f'_A\colon \fun{F}_1(A)\to \fun{F}_2(A)$ the induced
morphism. In order to see that $f_A=f'_A$, we start by proving that
there exists a third resolution
\begin{equation}\label{eqn:raff}
\cdots\to P_{i''_j}^{\oplus k''_j}\mor{d''_j}P_{i''_{j-1}}^{\oplus
k''_{j-1}}\mor{d''_{j-1}}\cdots\mor{d''_1}P_{i''_0}^{\oplus
k''_0}\mor{d''_0}A\to0
\end{equation}
and morphisms $s_j\colon P_{i''_j}^{\oplus k''_j}\to P_{i_j}^{\oplus
k_j}$ and $t_j\colon P_{i''_j}^{\oplus k''_j}\to P_{i'_j}^{\oplus
k'_j}$, for any $j\geq 0$, fitting into the following commutative
diagram:
\begin{equation}\label{eqn:diacomm}
\xymatrix{\cdots \ar[r]^-{d'_{j+1}} & P_{i'_j}^{\oplus k'_j}
\ar[r]^-{d'_j} & P_{i'_{j-1}}^{\oplus k'_{j-1}} \ar[r]^-{d'_{j-1}} &
\cdots \ar[r]^-{d'_1} & P_{i'_0}^{\oplus k'_0} \ar[dr]^-{d'_0} & \\
\cdots \ar[r]^-{d''_{j+1}} & P_{i''_j}^{\oplus k''_j} \ar[u]^{t_j}
\ar[d]_{s_j} \ar[r]^-{d''_j} & P_{i''_{j-1}}^{\oplus k''_{j-1}}
\ar[u]^{t_{j-1}} \ar[d]_{s_{j-1}} \ar[r]^-{d''_{j-1}} & \cdots
\ar[r]^-{d''_1} & P_{i''_0}^{\oplus k''_0} \ar[d]_{s_{0}}
\ar[u]^{t_{0}} \ar[r]^-{d''_0} & A. \\
\cdots \ar[r]_-{d_{j+1}} & P_{i_j}^{\oplus k_j} \ar[r]_-{d_j} &
P_{i_{j-1}}^{\oplus k_{j-1}} \ar[r]_-{d_{j-1}} & \cdots \ar[r]_-{d_1} &
P_{i_0}^{\oplus k_0} \ar[ur]_-{d_0} &}
\end{equation}
In fact, we can actually argue as follows. Let $F_0'':=P_{i'_0}^{\oplus k'_0}\times_A P_{i_0}^{\oplus k_0}$. By \eqref{ample2} of Definition
\ref{def:almostample}, there exists an epimorphism $d''_0\colon P_{i''_0}^{\oplus k''_0}\epi F_0''\epi A$ with maps $s_0$ and $t_0$ sitting, by definition,
in the commutative diagram
\[
\xymatrix{
P_{i'_0}^{\oplus k'_0} \ar[dr]^-{d'_0} & \\
P_{i''_0}^{\oplus k''_0} \ar[d]_{s_{0}}
\ar[u]^{t_{0}} \ar[r]^-{d''_0} & A. \\
P_{i_0}^{\oplus k_0} \ar[ur]_-{d_0} &}
\]
Let us now explain how we deal with the inductive step. Assume we have constructed the following commutative diagram
\[
\xymatrix{\cdots \ar[r]^-{d'_{j+1}} & P_{i'_j}^{\oplus k'_j}
\ar[r]^-{d'_j} & P_{i'_{j-1}}^{\oplus k'_{j-1}} \ar[r]^-{d'_{j-1}} &
\cdots \ar[r]^-{d'_1} & P_{i'_0}^{\oplus k'_0} \ar[dr]^-{d'_0} & \\
 &  & P_{i''_{j-1}}^{\oplus k''_{j-1}}
\ar[u]^{t_{j-1}} \ar[d]_{s_{j-1}} \ar[r]^-{d''_{j-1}} & \cdots
\ar[r]^-{d''_1} & P_{i''_0}^{\oplus k''_0} \ar[d]_{s_{0}}
\ar[u]^{t_{0}} \ar[r]^-{d''_0} & A. \\
\cdots \ar[r]_-{d_{j+1}} & P_{i_j}^{\oplus k_j} \ar[r]_-{d_j} &
P_{i_{j-1}}^{\oplus k_{j-1}} \ar[r]_-{d_{j-1}} & \cdots \ar[r]_-{d_1} &
P_{i_0}^{\oplus k_0} \ar[ur]_-{d_0} &}
\]
We can define $P_{i''_j}^{\oplus k''_j}$, $d''_j$, $s_j$ and $t_j$ as follows. Set $K'_{j-1}:=\ker d'_{j-1}$, $K''_{j-1}:=\ker d''_{j-1}$ and $K_{j-1}:=\ker d_{j-1}$. By definition, we have $d_{j-1}(s_{j-1}(K''_{j-1}))=d'_{j-1}(t_{j-1}(K''_{j-1}))=0$ and thus $s_{j-1}$ and $t_{j-1}$ induce two morphisms $r'_{j-1}\colon K''_{j-1}\to K'_{j-1}$ and $r_{j-1}\colon K''_{j-1}\to K_{j-1}$. Consider the fiber products
\[
F'_{j-1}:=P_{i'_j}^{\oplus k'_j}\times_{K'_{j-1}} K''_{j-1}\qquad F_{j-1}:=P_{i_j}^{\oplus k_j}\times_{K_{j-1}} K''_{j-1}
\]
along the maps $d'_j$, $r'_{j-1}$ and $d_j$, $r_{j-1}$ respectively. Since $d'_j$ and $d_j$ are epimorphisms onto $K'_{j-1}$ and $K_{j-1}$ respectively, we get epimorphisms $F'_{j-1}\epi K''_{j-1}$ and $\colon F_{j-1}\epi K''_{j-1}$. Take the fiber product $F''_{j-1}:=F'_{j-1}\times_{K''_{j-1}} F_{j-1}$ which then comes with an epimorphism onto $K''_{j-1}$. By \eqref{ample2} of Definition
\ref{def:almostample}, we get an epimorphism $d''_{j}\colon P_{i''_j}^{\oplus k''_j}\epi F_{j-1}''\epi K''_{j-1}$ with maps $s_j$ and $t_j$ making \eqref{eqn:diacomm} commutative.

Denoting by $f''_A\colon \fun{F}_1(A)\to \fun{F}_2(A)$ the morphism constructed
using \eqref{eqn:raff}, we get a diagram
\[
\xymatrix{\fun{F}_1(P_{i''_0}^{\oplus k''_0})\ar[rrrr]^-{\fun{F}_1(d''_0)}
\ar[dr]^-{\fun{F}_1(s_0)}\ar[ddd]_{f_{i''_0}^{\oplus k''_0}}& & & &\fun{F}_1(A)
\ar[dl]_-{\id}\ar[ddd]^{f''_A}\\
&\fun{F}_1(P_{i_0}^{\oplus k_0})\ar[rr]^-{\fun{F}_1(d_0)}\ar[d]_{f_{i_0}^{\oplus k_0}}
& &\fun{F}_1(A)\ar[d]_{f_A}&\\
&\fun{F}_2(P_{i_0}^{\oplus k_0})\ar[rr]_-{\fun{F}_2(d_0)}
& &\fun{F}_2(A)\ar@{}[ur]|{\bigstar}&\\
\fun{F}_2(P_{i''_0}^{\oplus k''_0})\ar[rrrr]_-{\fun{F}_2(d''_0)}
\ar[ur]_-{\fun{F}_2(s_0)}& & & &\fun{F}_2(A)\ar[ul]^-{\id}}
\]
where all squares and trapezoids but $\bigstar$ are commutative. Due to hypothesis
(ii) and Lemma \ref{lem:conv2} there exists a unique morphism
$\fun{F}_1(A)\to \fun{F}_2(A)$ making the following diagram commutative:
\[
\xymatrix{\fun{F}_1(P_{i''_0}^{\oplus k''_0})
\ar[d]_{\fun{F}_2(s_0)\comp f_{i''_0}^{\oplus k''_0}}
\ar[rr]^-{\fun{F}_1(d''_0)} & & \fun{F}_1(A) \ar[d] \\
\fun{F}_2(P_{i_0}^{\oplus k_0}) \ar[rr]_-{\fun{F}_2(d_0)} & & \fun{F}_2(A).}
\]
Since $\fun{F}_2(s_0)\comp f_{i''_0}^{\oplus k''_0}=f_{i_0}^{\oplus
k_0}\comp \fun{F}_1(s_0)$ by the naturality of $f$ on $\cat{C}$, both $f_A$ and $f''_A$ have this property
and then they coincide. Similarly one can prove that $f''_A=f'_A$.

Therefore $f_A=f'_A$, and so
$\tilde{f}(A):=f_A\colon\fun{F}_1(A)\to\fun{F}_2(A)$ is well defined
for every $A\in\cat{F}=\Db(\cat{E})\cap\cat{A}$. It is also easy to see that $\tilde{f}\colon
\fun{F}_1\rest{\cat{F}}\to\fun{F}_2\rest{\cat{F}}$ is a natural
transformation. Indeed, let $u\colon A\to B$ be a morphism of
$\cat{F}$. Consider a resolution of $B$
\begin{equation*}\label{eqn:res1}
\cdots\to P_{l_j}^{\oplus h_j}\mor{e_j}P_{l_{j-1}}^{\oplus
h_{j-1}}\mor{e_{j-1}}\cdots\mor{e_1}P_{l_0}^{\oplus h_0}\mor{e_0}B\to0,
\end{equation*}
where $l_j\in I$ and $h_j\in\NN$ for every $j\in\NN$.
We can find a resolution of $A$
\begin{equation*}\label{eqn:res2}
\cdots\to P_{i_j}^{\oplus k_j}\mor{d_j}P_{i_{j-1}}^{\oplus
k_{j-1}}\mor{d_{j-1}}\cdots\mor{d_1}P_{i_0}^{\oplus
k_0}\mor{d_0}A\to0
\end{equation*}
and morphisms $g_j\colon P_{i_{j}}^{\oplus k_j}\to P_{l_j}^{\oplus
h_j}$ defining a morphism of complexes compatible with $u$.
Indeed, set
$E_0:=A\times_B P_{l_0}^{\oplus h_0}$ to be the fiber product along
the morphisms $u$ and $e_0$. By \eqref{ample2} of Definition
\ref{def:almostample}, we get an epimorphism
$P_{i_0}^{\oplus k_0}\to E_0$ providing, by composition, the desired morphisms
$d_0$ and $g_0\colon P_{i_0}^{\oplus k_0}\to P_{l_0}^{\oplus
  h_0}$. Then we proceed by induction as in \eqref{eqn:diacomm}.

We can now consider the diagram
\[
\xymatrix{\fun{F}_1(P_{i_0}^{\oplus k_0}) \ar[rrrr]^-{\fun{F}_1(d_0)}
\ar[dr]^-{f_{i_0}^{\oplus k_0}} \ar[ddd]_{\fun{F}_1(g_0)} & & & &
\fun{F}_1(A) \ar[dl]_-{f_A} \ar[ddd]^{\fun{F}_1(u)} \\
& \fun{F}_2(P_{i_0}^{\oplus k_0}) \ar[rr]^-{\fun{F}_2(d_0)} \ar[d]_{\fun{F}_2(g_0)}
& & \fun{F}_2(A) \ar[d]_{\fun{F}_2(u)} & \\
& \fun{F}_2(P_{l_0}^{\oplus h_0}) \ar[rr]_-{\fun{F}_2(e_0)} & &
\fun{F}_2(B)\ar@{}[ur]|{\bigstar} & \\
\fun{F}_1(P_{l_0}^{\oplus h_0}) \ar[rrrr]_-{\fun{F}_1(e_0)}
\ar[ur]_-{f_{l_0}^{\oplus h_0}} & & & & \fun{F}_1(B)\ar[ul]^-{f_B}}
\]
where all squares and trapezoids but $\bigstar$ are commutative. Using the same argument as
above, we can take $m>N(A),N(B)$ and truncate the resolutions of $A$
and $B$ at step $m$. Then, applying (ii)
and Lemma \ref{lem:conv2}, we see that there is a unique morphism
$\fun{F}_1(A)\to\fun{F}_2(B)$ completing the following diagram to a commutative
square
\[
\xymatrix{\fun{F}_1(P_{i_0}^{\oplus k_0})
\ar[d]_{\fun{F}_2(g_0)\comp f_{i_0}^{\oplus k_0}}
\ar[rr]^-{\fun{F}_1(d_0)} & &
\fun{F}_1(A)\ar[d] \\
\fun{F}_2(P_{l_0}^{\oplus h_0}) \ar[rr]_-{\fun{F}_2(e_0)} & &
\fun{F}_2(B).}
\]
Since $\fun{F}_2(g_0)\comp f_{i_0}^{\oplus k_0}=f_{l_0}^{\oplus
h_0}\comp \fun{F}_1(g_0)$, both $\fun{F}_2(u)\comp f_A$ and $f_B\comp
\fun{F}_1(u)$ have this property. It follows that
$\fun{F}_2(u)\comp\tilde{f}(A)=\tilde{f}(B)\comp\fun{F}_1(u)$,
thus proving that $\tilde{f}\colon\fun{F}_1\rest{\cat{F}}\to
\fun{F}_2\rest{\cat{F}}$ is a natural transformation. It is clear by
construction that $\tilde{f}\rest{\cat{C}}=f$ and that $\tilde{f}$ is
unique with this property.

Since the objects of $\cat{D}_0$ are precisely (up to isomorphism)
shifts of objects of $\cat{F}$, we just need to define $f_0(A[k])$ for
$A\in\cat{F}$ and $k\in\ZZ$. Of course, we must set
$f_0(A[k]):=\tilde{f}(A)[k]$, and we have just to show that
$f_0(B[k])\comp\fun{F}_1(u)=\fun{F}_2(u)\comp f_0(A)$ for all
objects $A,B\in\cat{F}$ and every $u\in\Hom(A,B[k])$. Notice that, for
simplicity, we are not making explicit the natural isomorphism
$\fun{F}_i(A[k])\iso\fun{F}_i(A)[k]$, for $i=1,2$.

Now, there is
nothing to prove if $k<0$ (because then $u=0$) or $k=0$ (because we
have already seen that $\tilde{f}$ is a natural transformation), so we
assume $k>0$. Actually we can reduce to the case $k=1$, thanks to the
fact that one can always factor $u$ as $u=u_k\comp\cdots\comp u_1$,
where (for $j=1,\dots,k$) $u_j\in\Hom(C_{j-1}[j-1],C_j[j])$ and
$C_0=A,\dots,C_k=B$ are objects of $\cat{F}$ (see Step 4 in the proof
of \cite[Prop.\ B.1]{LO}). Now, completing $u$ to a distinguished
triangle $B\to C\mor{v}A\mor{u}B[1]$, $C$ is again an object of
$\cat{F}$. Then by axiom (TR3) there exists a morphism
$h\colon\fun{F}_1(A)\to\fun{F}_2(A)$ such that the diagram
\[
\xymatrix{\fun{F}_1(B) \ar[rr] \ar[d]^{\tilde{f}(B)} & & \fun{F}_1(C)
\ar[rr]^-{\fun{F}_1(v)} \ar[d]^{\tilde{f}(C)} & & \fun{F}_1(A) \ar[d]^h
\ar[rr]^-{\fun{F}_1(u)} & & \fun{F}_1(B[1]) \ar[d]^{f_0(B[1])} \\
\fun{F}_2(B) \ar[rr] & & \fun{F}_2(C) \ar[rr]^-{\fun{F}_2(v)}
& & \fun{F}_2(A) \ar[rr]^-{\fun{F}_2(u)} & & \fun{F}_2(B[1])}
\]
commutes. Since $\Hom(\fun{F}_1(B[1]),\fun{F}_2(A))=0$ by hypothesis,
$h$ is the unique morphism such that
$h\comp\fun{F}_1(v)=\fun{F}_2(v)\comp\tilde{f}(C)$. Hence
$h=\tilde{f}(A)=f_0(A)$, and we conclude that
$f_0(B[1])\comp\fun{F}_1(u)=\fun{F}_2(u)\comp f_0(A)$.
\end{proof}

Let us specialize to the case of isomorphisms.

\begin{cor}\label{cor:critiso1}
With the same hypotheses on $\cat{E}$, $\cat{A}$ and $\cat{T}$,
let $\fun{F}_1,\fun{F}_2\colon\Db(\cat{E})\to\cat{T}$ be exact
functors and let
$f\colon\fun{F}_1\rest{\cat{C}}\isomor\fun{F}_2\rest{\cat{C}}$ be an
isomorphism of functors. Assume that $\fun{F}_1$ satisfies condition
$(1)$ of $\Ca$. Then there exists a unique isomorphism compatible with
shifts
$f_0\colon\fun{F}_1\rest{\cat{D}_0}\isomor\fun{F}_2\rest{\cat{D}_0}$
extending $f$.
\end{cor}

\begin{proof}
Since $f$ is an isomorphism, we can use Lemma \ref{lem:conv1} in the
above argument to show that there is an isomorphism
$\fun{F}_1(A)\iso\fun{F}_2(A)$, for all
$A\in\Db(\cat{E})\cap\cat{A}$. Thus (ii) in Proposition
\ref{prop:extab} follows from $(1)$ in $\Ca$. Therefore, by the above
proposition, there is a unique natural transformation compatible with
shifts $f_0\colon\fun{F}_1\to\fun{F}_2$ extending $f$. Uniqueness,
applied also to $f^{-1}$, $f\comp f^{-1}$ and $f^{-1}\comp f$,
immediately implies that $f_0$ is an isomorphism
\end{proof}

\subsection{The criterion: extension to the whole derived category}\label{subsec:crit2}

In order to extend the natural transformation $f$ of Proposition
\ref{prop:extab} to $\Db(\cat{E})$, the exact functors have to satisfy
one more assumption.

\begin{prop}\label{prop:extending1}
Let $\cat{T}$ be a triangulated category and let $\cat{E}$ be a full
exact subcategory of an abelian category $\cat{A}$
satisfying {\rm (E1)}, {\rm (E2)} and {\rm (E3)}.  Let
$\fun{F}_1,\fun{F}_2\colon\Db(\cat{E})\to\cat{T}$ be exact functors
with a natural transformation
$f\colon\fun{F}_1\rest{\cat{C}}\to\fun{F}_2\rest{\cat{C}}$. Assume
moreover the following:
\begin{itemize}
\item[(i)] $\fun{F}_1$ and $\fun{F}_2$ both satisfy condition $(1)$
of $\Ca$;
\item[(ii)] For any $A,B\in\Db(\cat{E})\cap\cat{A}$ and any integer $k<0$
\[
\Hom(\fun{F}_1(A),\fun{F}_2(B)[k])=0.
\]
\item[(iii)] for all $C\in\Db(\cat{E})$ with trivial cohomologies in
positive degrees, there is $i\in I$ such that
\[
\Hom(\fun{F}_1(C),\fun{F}_2(P_i))=0
\]
and $i$ satisfies condition \eqref{ample2} of
Definition \ref{def:almostample} for $H^0(C)$.
\end{itemize}
Then there exists a unique natural transformation of exact functors
$g\colon\fun{F}_1\to\fun{F}_2$ extending $f$.
\end{prop}

\begin{proof}
For $n\in\NN$, denote by $\cat{D}_n$ the (strictly) full subcategory
of $\Db(\cat{E})$ with objects the complexes $A$ with the following
property: there exists $a\in\ZZ$ such that $H^p(A)=0$ for $p<a$ or
$p>a+n$. We are going to prove by induction on $n$ that $f$ extends
uniquely to a natural transformation compatible with shifts
$f_n\colon\fun{F}_1\rest{\cat{D}_n}\to\fun{F}_2\rest{\cat{D}_n}$. Once
we do this, it is obvious that for every object $A$ of $\Db(\cat{E})$
we can define $g(A):=f_n(A)$ if $A\in\cat{D}_n$, and that $g$ is then
the unique required extension of $f$.

The case $n=0$ having already been proved in Proposition
\ref{prop:extab}, we come to the inductive step from $n-1$ to
$n>0$. For every object $A\in\cat{D}_n$ we need to define
$f_n(A)\colon\fun{F}_1(A)\to\fun{F}_2(A)$. To this purpose, we can
assume that $H^p(A)=0$ for $p<-n$ or $p>0$. If $A=\{\cdots\to
A^0\mor{d^0}A^1\to\cdots\}$, let
\[
s\colon P_i^{\oplus k}\epi\ker d^0
\]
(for some $i\in I$ and $k\in\NN$) be an epimorphism such that
$\Hom(\fun{F}_1(A),\fun{F}_2(P_i))=0$. Notice that $s$ can be found as
follows: after choosing an epimorphism $P_j^{\oplus l}\epi\ker d^0$
(with $j\in I$ and $l\in\NN$), take $i\in I$ which satisfies
condition (iii) for $A\oplus P_j^{\oplus l}$ (so that, in particular,
there is an epimorphism $P_i^{\oplus k}\epi P_j^{\oplus l}$), and
define $s$ to be the composition $P_i^{\oplus k}\epi P_j^{\oplus
  l}\epi\ker d^0$. Denoting by $t\colon P_i^{\oplus k}\to A$ the
composition of $s$ with the natural morphism $\ker d^0\to A$, it is
then clear that $H^0(t)\colon P_i^{\oplus k}\to H^0(A)$ is an
epimorphism. It follows that we have a distinguished triangle
\begin{equation}\label{eqn:dist}
C[-1]\to P_i^{\oplus k}\mor{t}A\mor{t_1}C
\end{equation}
with $C\in\cat{D}_{n-1}$. Hence, by the inductive hypothesis and using
axiom (TR3), we obtain a commutative diagram whose rows are
distinguished triangles
\[
\xymatrix{\fun{F}_1(C)[-1] \ar[rr] \ar[d]^{f_{n-1}(C)[-1]} & &
\fun{F}_1(P_{i}^{\oplus k}) \ar[rr]^-{\fun{F}_1(t)}
\ar[d]^{f_{n-1}(P_i^{\oplus k})} & & \fun{F}_1(A) \ar[d]^{f_A}
\ar[rr]^-{\fun{F}_1(t_1)} & & \fun{F}_1(C) \ar[d]^{f_{n-1}(C)} \\
\fun{F}_2(C)[-1] \ar[rr] & & \fun{F}_2(P_{i}^{\oplus k})
\ar[rr]^-{\fun{F}_2(t)} & & \fun{F}_2(A) \ar[rr]^-{\fun{F}_2(t_1)}
& & \fun{F}_2(C)}
\]
for some $f_A\colon\fun{F}_1(A)\to\fun{F}_2(A)$. Observe that, since
$\Hom(\fun{F}_1(A),\fun{F}_2(P_i^{\oplus k}))=0$ by assumption, $f_A$
is the unique morphism such that the square on the right commutes.

In order to prove that $f_A$ does not depend on the choice of
$s$, assume that $s'\colon P_{i'}^{\oplus k'}\epi\ker d^0$ is
another epimorphism such that
$\Hom(\fun{F}_1(A),\fun{F}_2(P_{i'}))=0$, and thus inducing another
morphism $f'_A\colon\fun{F}_1(A)\to\fun{F}_2(A)$. We claim that we can
find a third epimorphism $s''\colon P_{i''}^{\oplus k''}\epi\ker d^0$
such that $\Hom(\fun{F}_1(A),\fun{F}_2(P_{i''}))=0$ (inducing
$f''_A\colon\fun{F}_1(A)\to\fun{F}_2(A)$) and fitting into a
commutative diagram
\[
\xymatrix{P_{i''}^{\oplus k''} \ar[d]_{w'} \ar[rr]^-{w}
\ar[drr]^-{s''} & & P_i^{\oplus k} \ar[d]^{s} \\
P_{i'}^{\oplus k'} \ar[rr]_-{s'} & &
\ker d^0.}
\]
This can be easily seen if one takes $i''\in I$ satisfying condition
(iii) for $A\oplus P_j^{\oplus l}$, where $j\in I$ and $l\in\NN$ are
such that there exists an epimorphism $P_j^{\oplus l}\epi P_i^{\oplus
k}\times_{\ker d^0}P_{i'}^{\oplus k'}$ (then $s''$ is an epimorphism
because the natural map $ P_i^{\oplus k}\times_{\ker
d^0}P_{i'}^{\oplus k'}\epi\ker d^0$ is an epimorphism). Observing
that the morphisms $t'\colon P_{i'}^{\oplus k'}\to A$ and $t''\colon
P_{i''}^{\oplus k''}\to A$ (induced, respectively, by $s'$ and $s''$)
obviously satisfy $t\comp w=t''=t'\comp w'$, by axiom (TR3) there is a
commutative diagram whose rows are distinguished triangles
\[
\xymatrix{P_{i''}^{\oplus k''} \ar[rr]^-{t''} \ar[d]^w & &
A \ar[rr]^-{t''_1} \ar[d]^{\id} & & C'' \ar[rr] \ar[d]^v & &
P_{i''}^{\oplus k''}[1] \ar[d]^{w[1]} \\
P_i^{\oplus k} \ar[rr]^-t & & A \ar[rr]^-{t_1} & & C \ar[rr] & &
P_i^{\oplus k}[1]}
\]
for some $v\colon C''\to C$. As the diagram
\[
\xymatrix{\fun{F}_1(A) \ar[d]^{f''_A} \ar[rr]^-{\fun{F}_1(t''_1)} & &
\fun{F}_1(C'') \ar[d]^{f_{n-1}(C'')} \ar[rr]^-{\fun{F}_1(v)} & &
\fun{F}_1(C) \ar[d]^{f_{n-1}(C)} \\
\fun{F}_2(A) \ar[rr]^-{\fun{F}_2(t''_1)} & & \fun{F}_2(C'')
\ar[rr]^-{\fun{F}_2(v)} & & \fun{F}_2(C)}
\]
commutes (the square on the left by definition of $f''_A$, the square
on the right because $f_{n-1}$ is a natural transformation by
induction) and since $v\comp t''_1=t_1$, we obtain
\[
f_{n-1}(C)\comp\fun{F}_1(t_1)=
f_{n-1}(C)\comp\fun{F}_1(v)\comp\fun{F}_1(t''_1)=
\fun{F}_2(v)\comp\fun{F}_2(t''_1)\comp f''_A=
\fun{F}_2(t_1)\comp f''_A.
\]
On the other hand, $f_A$ is the only morphism with the property that
$f_{n-1}(C)\comp\fun{F}_1(t_1)=\fun{F}_2(t_1)\comp f_A$. It follows
that $f_A=f''_A$ and similarly $f'_A=f''_A$, thereby proving that
$f_A=f'_A$. Therefore we can set $f_n(A):=f_A$, and more generally
$f_n(A[k]):=f_A[k]$ for every integer $k$, thus defining $f_n$ on
every object of $\cat{D}_n$.

To conclude the inductive step it is enough to show that $f_n$ is a
natural transformation, because then it is clear by definition that
$f_n$ is compatible with shifts, that
$f_n\rest{\cat{C}}=f_{n-1}\rest{\cat{C}}=f$ (actually also
$f_n\rest{\cat{D}_{n-1}}=f_{n-1}$) and that $f_n$ is unique with these
properties. So we have to prove that
\begin{equation}\label{eqn:nat}
f_n(B)\comp\fun{F}_1(u)=\fun{F}_2(u)\comp f_n(A)
\end{equation}
for every morphism $u\colon A\to B$ of $\cat{D}_n$. If \eqref{eqn:nat}
holds for a morphism $u$ of $\cat{D}_n$, we say that $f_n$ is
\emph{compatible} with $u$.

Recall that in
$\D[b](\cat{E})$ we can write $u=w_1\comp w_2^{-1}$, where $w_1$ and $w_2$ are
(represented by) morphisms of complexes and $w_2$ is a
quasi--isomorphism (hence $w_1$ and $w_2$ are again in $\cat{D}_n$). Thus
$f_n$ is compatible with $u$ if it is
compatible both with $w_2^{-1}$ (or, equivalently, with $w_2$) and with
$w_1$. In other words, it is harmless to assume directly that $u$ is a
morphism of complexes, denoted by $A=\{\cdots\to
A^0\mor{d^0}A^1\to\cdots\}$ and $B=\{\cdots\to
B^0\mor{e^0}B^1\to\cdots\}$. We can also assume that, as before,
$H^p(A)=0$ for $p<-n$ or $p>0$. Moreover, we denote by $c$ the
greatest integer such that $H^c(B)\ne0$ (of course, if $B\iso0$ there
is nothing to prove). Now our aim is to show that the problem of
verifying \eqref{eqn:nat} can be reduced to a similar problem with
another ``simpler'' morphism in place of $u$. To this purpose we
distinguish two cases according to the value of $c$.

\smallskip

\noindent{\bf Case 1: $c<0$.} Choose $j\in I$ which satisfies
\eqref{ample2} of Definition \ref{def:almostample} for $K:=\ker
d^0\times_{\ker e^0}B^{-1}$, let $P_j^{\oplus l}\epi K$ be an
epimorphism, and take $i\in I$ satisfying condition (iii)
for $A\oplus P_j^{\oplus l}$. Then, reasoning as before, we get an
epimorphism $s\colon P_i^{\oplus k}\epi\ker d^0$ (notice that
$K\to\ker d^0$ is an epimorphism because $B^{-1}\to\ker e^0=\im
e^{-1}$ is an epimorphism, as $c<0$) which can be used to
define $f_A$. Moreover, denoting by $t\colon P_i^{\oplus k}\to A$ the
morphism (of complexes) induced by $s$, it is clear that $u\comp t$
is given by a morphism $w\colon P_i^{\oplus k}\to\ker
e^0\subseteq B^0$ which factors through the natural morphism $K\to\ker
e^0$. In particular, there exists $w'\colon P_i^{\oplus k}\to B^{-1}$
such that $w=e^{-1}\comp w'$. This proves that
$u\comp t$ is homotopic to $0$, whence it is $0$ in
$\D[b](\cat{E})$. From this and from the distinguished triangle
\eqref{eqn:dist} it follows that $u=v\comp t_1$ for some $v\colon C\to
B$ (with $C\in\cat{D}_{n-1}$). As $f_n$ is compatible with $t_1$ by
definition of $f_A=f_n(A)$, in order to check \eqref{eqn:nat} it is
therefore enough to show that $f_n$ is compatible with $v$. Notice
that, if $A\in\cat{D}_m$ for some $0<m\le n$, then
$C\in\cat{D}_{m-1}$. On the other hand, if $A\in\cat{D}_0$ (hence $A$
is isomorphic to an object of $\cat{F}$), then $C\in\cat{D}_0$ and
$C[-1]$ is isomorphic to an object of $\cat{F}$. So in this last case,
passing from $u$ to $v[-1]$, $c$ increases by $1$.

\smallskip

\noindent{\bf Case 2: $c\ge0$.} Choose an epimorphism $P_j^{\oplus l}\epi\ker e^c$ (with
$j\in I$ and $l\in\NN$) and take $i\in I$ satisfying
condition (iii) for $A[c]\oplus B[c]\oplus P_j^{\oplus l}$. Then, as
usual, we can find an epimorphism $s'\colon P_i^{\oplus k}\epi\ker
e^c$ which can be used to define $f_{B[c]}$. Denoting by $t'\colon
P_i^{\oplus k}[-c]\to B$ the morphism induced by $s'$, and extending
it to a distinguished triangle
\[
C'[-1]\to P_i^{\oplus k}[-c]\mor{t'}B\mor{t'_1}C'
\]
(with $C'\in\cat{D}_{n-1}$), we claim that \eqref{eqn:nat} follows
once one proves that $f_n$ is compatible with $v':=t'_1\comp u\colon
A\to C'$. To see this, observe that in the diagram
\[
\xymatrix{\fun{F}_1(A) \ar[d]^{f_n(A)} \ar[rr]^-{\fun{F}_1(u)} & &
\fun{F}_1(B) \ar[d]^{f_n(B)} \ar[rr]^-{\fun{F}_1(t'_1)} & &
\fun{F}_1(C') \ar[d]^{f_n(C')} \\
\fun{F}_2(A) \ar[rr]^-{\fun{F}_2(u)} & & \fun{F}_2(B)
\ar[rr]^-{\fun{F}_2(t'_1)} & & \fun{F}_2(C')}
\]
the square on the right commutes by definition of
$f_{B[c]}[-c]=f_n(B)$, whence (assuming compatibility of $f_n$ with
$v'$)
\[
\fun{F}_2(t'_1)\comp
(f_n(B)\comp\fun{F}_1(u)-\fun{F}_2(u)\comp f_n(A))=
f_n(C')\comp\fun{F}_1(v')-\fun{F}_2(v')\comp f_n(A)=0.
\]
It follows that $f_n(B)\comp\fun{F}_1(u)-\fun{F}_2(u)\comp f_n(A)$
factors through $\fun{F}_2(t')$, and then it must be $0$ (which means
that \eqref{eqn:nat} holds) because
$\Hom(\fun{F}_1(A),\fun{F}_2(P_i^{\oplus k}[-c]))=0$ by the choice of
$i$. Observe that, similarly as above, if $B\in\cat{D}_m$ for some
$0<m\le n$, then $C'\in\cat{D}_{m-1}$, whereas, if $B\in\cat{D}_0$,
then $C'\in\cat{D}_0$ and, passing from $u$ to $v'$, $c$ decreases by
$1$.

\smallskip

To finish the proof, note that, applying the above procedure, one
obtains a morphism having the source in $\cat{D}_{n-1}$ and the same
target (in Case 1) or the same source and the target in
$\cat{D}_{n-1}$ (in Case 2). So it is enough to show that, repeating
the procedure a sufficient number of times, one necessarily encounters
both cases (because then one reduces to check compatibility of $f_n$
with a morphism of $\cat{D}_{n-1}$, where it holds by
induction). Indeed, if one always encounters Case 1 (the argument is
completely similar for Case 2), then in a finite number (at most $n$)
of steps the source becomes an object of $\cat{D}_0$. Applying another
finite number of steps, one eventually gets $c=0$, namely Case 2.
\end{proof}

In the paper we will need the following special case of the above result.

\begin{cor}\label{cor:critiso2}
With the same hypotheses on $\cat{E}$, $\cat{A}$ and $\cat{T}$,
let $\fun{F}_1,\fun{F}_2\colon\Db(\cat{E})\to\cat{T}$ be exact
functors and let
$f\colon\fun{F}_1\rest{\cat{C}}\isomor\fun{F}_2\rest{\cat{C}}$ be an
isomorphism. Assume moreover that $\fun{F}_1$ satisfies $\Ca$.  Then
there exists a unique isomorphism of exact functors
$g\colon\fun{F}_1\isomor\fun{F}_2$ extending $f$.
\end{cor}

\begin{proof}
As $f$ is an isomorphism, we can apply Corollary \ref{cor:critiso1} so
that $\fun{F}_1(A)\iso\fun{F}_2(A)$, for all
$A\in\Db(\cat{E})\cap\cat{A}$. Hence hypotheses (i) and (ii) in Proposition
\ref{prop:extending1} follow from $\Ca$. Analogously, for (iii) we use
that $\fun{F}_1(P_i)\iso\fun{F}_2(P_i)$ by assumption. Thus
Proposition \ref{prop:extending1} applies and we get a unique natural
transformation of exact functors $g\colon\fun{F}_1\to\fun{F}_2$
extending $f$. The fact that $g$ is an isomorphism is again a formal
consequence of uniqueness, as in the proof of Corollary
\ref{cor:critiso1}.
\end{proof}

In the case $\cat{E}=\cat{A}$, we are going to give a sufficient
condition under which $\Ca$ is automatically satisfied. We leave it to
the reader to formulate a similar statement which ensures that the
hypotheses of Proposition \ref{prop:extending1} are satisfied.

\begin{lem}\label{lem:adjoint}
Let $\fun{F}\colon\Db(\cat{A})\to\cat{T}$ be an exact functor
admitting a left adjoint and satisfying $\Cb$. Assume moreover that
$\{P_i\}_{i\in I}$ is an almost ample set in $\cat{A}$. Then $\fun{F}$
satisfies $\Ca$ as well.
\end{lem}

\begin{proof}
Observing that part $(1)$ of $\Ca$ coincides with $\Cb$ because
$\cat{E}=\cat{A}$, it remains to prove part $(2)$ of $\Ca$. Denoting
by $\fun{F}^*\colon\cat{T}\to\Db(\cat{A})$ the left adjoint of
$\fun{F}$, we claim that $H^p(\fun{F}^*\comp\fun{F}(A))=0$ for any
$A\in\cat{A}$ and for any $p>0$. Indeed, otherwise there would exist
$A\in\cat{A}$ and $m>0$ with a non-zero morphism
$\fun{F}^*\comp\fun{F}(A)\to H^m(\fun{F}^*\comp\fun{F}(A))[-m]$ (it is
enough to let $m$ be the largest integer such that
$H^m(\fun{F}^*\comp\fun{F}(A))\ne0$). But then
\[
0\neq\Hom(\fun{F}^*\comp\fun{F}(A),H^m(\fun{F}^*\comp\fun{F}(A))[-m])
\iso\Hom(\fun{F}(A),\fun{F}(H^m(\fun{F}^*\comp\fun{F}(A)))[-m])
\]
by adjunction, contradicting $\Cb$.

The above implies more generally that
$H^p(\fun{F}^*\comp\fun{F}(C))=0$, for any $C\in\Db(\cat{A})$ having
trivial cohomologies in positive degrees and any $p>0$. To see this,
we can proceed by induction on the smallest integer $n$ such that $C$
is an object of $\cat{D}_n$ (the full subcategory of
$\Db(\cat{E})=\Db(\cat{A})$ defined at the beginning of the proof of
Proposition \ref{prop:extending1}). Indeed, we can assume without loss
of generality that $H^0(C)$ is the greatest non-trivial cohomology of
$C$. Then $C$ is isomorphic to an object of $\cat{A}$ if $n=0$, so the
statement has already been proved in this case. If $n>0$ we have a
distinguished triangle
\[
C'\to C\to H^0(C)\to C'[1]
\]
with $C'$ having non-trivial cohomologies only in negative degrees and
$C'\in\cat{D}_{n-1}$. By induction, for $p>0$, we have
$H^p(\fun{F}^*\comp\fun{F}(C'))=H^p(\fun{F}^*\comp\fun{F}(H^0(C)))=0$. Thus
$H^p(\fun{F}^*\comp\fun{F}(C))=0$.

Then
for such an object $C$ and for any $i\in I$ we have
\[
\Hom(\fun{F}(C),\fun{F}(P_i))\iso\Hom(\fun{F}^*\comp\fun{F}(C),P_i)
\iso\Hom(H^0(\fun{F}^*\comp\fun{F}(C)),P_i).
\]
For the last isomorphism above, take the distinguished triangle
\[
C''\to\fun{F}^*\comp\fun{F}(C)\to H^0(\fun{F}^*\comp\fun{F}(C))\to C''[1],
\]
where again, $C''$ has cohomologies in degrees smaller than zero. Hence
\[
\Hom(C'',P_i)\iso\Hom(C''[1],P_i)\iso 0.
\]

Therefore part $(2)$ of $\Ca$ is satisfied if one takes $i\in I$ as in
Definition \ref{def:almostample} for
$H^0(\fun{F}^*\comp\fun{F}(C))\oplus H^0(C)$.
\end{proof}

Combining the above result with Corollary \ref{cor:critiso2}
immediately gives a proof of Proposition \ref{prop:extending}.

\subsection{The geometric case and some examples}\label{subsec:smoothex}

In this section we want to clarify which abelian category $\cat{A}$
and exact subcategory $\cat{E}$ have to be taken in order to use the
results in Section \ref{subsec:crit2} to prove Theorems \ref{thm:main2}
and \ref{thm:main1}.

Therefore let $X$ be a quasi-projective
scheme and let $Z$ be a projective subscheme of $X$. Assume further
that $\ko_{iZ}\in\Dp(X)$ for all $i>0$. Set
\[
\cat{A}:=\Coh_Z(X)\qquad\cat{E}:=\Dp[Z](X)\cap\Coh_Z(X).
\]

\begin{prop}\label{prop:excat}
Under the above assumptions, $\cat{E}$ is a full exact subcategory of
$\cat{A}$, {\rm (E1)}, {\rm (E2)} and {\rm (E3)} are satisfied and
$\Dp[Z](X)=\Db(\cat{E})\subseteq\Db(\cat{A})$.
\end{prop}

\begin{proof}
The subcategory $\cat{E}$ is closed under extensions, hence $\cat{E}$
is a full exact subcategory of $\cat{A}$ (see
\cite[Sect.\ 4]{K2}). Condition (E1) follows from the fact that, if
$f$ is an admissible epimorphism in $\cat{E}$, then $\ker
f\in\cat{E}$. As $\ko_{iZ}\in\Dp(X)$ for all $i>0$, (E2) holds true
taking $\{P_i\}_{i\in I}=\cat{Amp}(Z,X,H)$ defined in
\eqref{eqn:almamp} (with $H$ an ample divisor on $X$).

Obviously $\Db(\cat{E})$ is a full subcategory of $\Dp[Z](X)$. To show
that they are actually equal, one has to apply an induction argument
similar to the one in the first part of the proof of Proposition
\ref{prop:extending1}. To give a hint, let $\cat{D}_n$ be the
(strictly) full subcategory of $\Dp[Z](X)$ with objects the complexes
$\ka$ with the following property: there exists $a\in\ZZ$ such that
$H^p(\ka)=0$ for $p<a$ or $p>a+n$. Given $\ka\in\Dp[Z](X)$, there
exists $n\ge0$ such that $\ka\in\cat{D}_n$, and one can prove that
$\ka\in\Db(\cat{E})$ by induction on $n$. Indeed, if $n=0$, there is
nothing to prove. Otherwise we can assume without loss of generality
that $H^p(\ka)=0$ for $p<-n$ or $p>0$. Then $\ka$ sits in a
distinguished triangle
\[
\kc[-1]\to P_i^{\oplus k}\to\ka\to\kc,
\]
where $P_i\in\cat{Amp}(Z,X,H)$, $k\in\NN$ and $\kc\in\cat{D}_{n-1}$.

As for (E3), one can prove more generally that for every
$\ka\in\Db(\cat{E})=\Dp[Z](X)$ there exists an integer $N(\ka)$ such
that $\Hom_{\Db(\cat{A})}(\ka,\kb[i])=0$, for every $i>N(\ka)$ and
every $\kb\in\cat{A}$. Indeed, this follows from the isomorphism
\[
\Hom_{\Db(\cat{A})}(\ka,\kb[i])\iso
\Hom_{\Db(X)}(\ko_X,\ka\dual\otimes\kb[i]),
\]
which holds because $\ka$ is perfect. More precisely, being the
cohomologies of $\ka\dual\otimes\kb$ bounded with bound depending
only on $\ka$, the vanishing of
$\Hom_{\Db(X)}(\ko_X,\ka\dual\otimes\kb[i])$, for $i>N(\ka)$, can be
deduced by induction on the cohomologies of $\ka\dual\otimes\kb$ using
Grothendieck vanishing theorem (see, for example, \cite[III, Thm.\ 2.7]{Ha}).
\end{proof}

\begin{remark}\label{rmK:2cond}
In view of Proposition \ref{prop:amp}, it is easy to see that, if $X$,
$Z$, $\cat{E}$ and $\cat{A}$ are as above, then condition $\Cc$ in the
introduction implies $\Ca$. Indeed, in this case (1) in $\Cc$ and
$\Ca$ coincide. As for (2), consider $C\in\Db(\cat{E})$ with trivial cohomologies in positive degrees. Then, by Proposition \ref{prop:amp}, there is an integer $N_1$ such that for all $i<N_1$ and $j\ll i$ part \eqref{ample2} of Definition \ref{def:almostample} holds true for $H^0(C)$ and $P_k:=\ko_{|i|Z}(jH)\in\cat{Amp}(Z,X,H)$, where $k=(i,j)$. By (2) in $\Cc$, we can take another integer $N_2$ such that, for $i'<N_2$ and $j'\ll i'$,
\[
\Hom(\fun{F}(C),\fun{F}(P_{k'}))=0,
\]
where $k'=(i',j')$. Considering $\min\{N_1,N_2\}$, this shows that (2) of $\Ca$ holds as well. Therefore, in the proof of Theorems
\ref{thm:main2} and \ref{thm:main1} we can freely use the results in
Section \ref{subsec:crit2}.
\end{remark}

It may be useful to keep in mind some examples of exact functors
satisfying $\Cc$.

\begin{ex}\label{ex:condfun}
In this example we assume that $X_1$ is a quasi-projective scheme with a projective subscheme $Z_1$ such that $\ko_{iZ_1}\in\Dp(X_1)$, for all $i>0$, and $T_0(\ko_{Z_1})=0$.
	
(i) Using \eqref{ample3} in Definition \ref{def:almostample} (which holds
thanks to Proposition \ref{prop:amp}), it is very easy to verify that full functors $\fun{F}\colon\Dp[Z_1](X_1)\to\Dp[Z_2](X_2)$ satisfy $\Cc$ for any scheme $X_2$ containing a closed subscheme $Z_2$.

(ii) For the same reason, a trivial example of a functor with
  the property $\Cc$ but which is not full is $\id\oplus\id\colon\Dp[Z_1](X_1)\to\Dp[Z_1](X_1)$.

(iii) Following the same argument as in \cite[Sect.\ 4]{CS}, in the supported
  setting one may take exact functors $\Db_{Z_1}(X_1)\to\Db_{Z_2}(X_2)$
  induced by exact full functors $\Coh_{Z_1}(X_1)\to\Coh_{Z_2}(X_2)$,
  where $X_1$ and $X_2$ are smooth quasi-projective varieties. These
  functors obviously satisfy $\Cc$.
\end{ex}

We conclude this section with the following easy result making clear
that in the smooth case without support conditions, $\Cb$ is
equivalent to $\Cc$ in the introduction.

\begin{prop}\label{prop:smooth}
Let $X_1$ be a smooth projective scheme such that $\dim(X_1)>0$ and let $X_2$ be a scheme containing
a closed subscheme $Z_2$. Then an exact functor
$\fun{F}\colon\Db(X_1)\to\Db_{Z_2}(X_2)$ satisfies $\Cc$ if and only if it
satisfies $\Cb$.
\end{prop}

\begin{proof}
Clearly it is enough to show that $\Cb$ implies (2) in $\Cc$. Since
$Z_1=X_1$, $\ko_{|i|Z_1}(jH_1)=\ko_{X_1}(jH_1)$, for all
$i,j\in\ZZ$, and when $j$ varies they give rise to an almost ample set. Hence it is enough to show that for any $\ka\in\Db(X_1)$
with trivial cohomologies in positive degrees, there is
$N\in\ZZ$ such that $\Hom(\fun{F}(\ka),\fun{F}(\ko_{X_1}(iH_1)))=0$
for any $i<N$. Observing that $\fun{F}$ has a left adjoint $\fun{F}^*$
by \cite{BB} (see also \cite[Rmk.\ 2.1]{CS}), the same argument as in
the proof of Lemma \ref{lem:adjoint} ensures that
$\fun{F}^*\comp\fun{F}(\ka)$ has trivial cohomologies in positive degrees.

Assume, without loss of generality that $H^0(\fun{F}^*\comp\fun{F}(\ka))$ is the last non-trivial cohomology. Then we have a distinguished triangle
\[
\ka'\to\fun{F}^*\comp\fun{F}(\ka)\to H^0(\fun{F}^*\comp\fun{F}(\ka))\to\ka'[1]
\]
where $\ka'$ has non-trivial cohomologies only in negative degrees. Hence $\Hom(\ka',\kb)=0$, for all $\kb\in\Coh(X_1)$.

By the last statement in Proposition \ref{prop:amp} (and \eqref{ample3} of Definition \ref{def:almostample}), there is an integer $N$ such that, for all $i<N$ we have $\Hom(H^0(\fun{F}^*\comp\fun{F}(\ka)),\ko_{X_1}(iN_1))=0$. By the above remark and adjunction, we then have
\[
0=\Hom(\fun{F}^*\comp\fun{F}(\ka),\ko_{X_1}(iN_1))\iso\Hom(\fun{F}(\ka),\fun{F}(\ko_{X_1}(iN_1))),
\]
for all $i<N$. This is precisely (2) in $\Cc$.
\end{proof}

\begin{ex}\label{ex:smooth}
In view of Proposition \ref{prop:smooth} and of \cite[Prop.\ 5.1]{CS}, a non-trivial class of exact functors satisfying $\Cc$ is provided by the functors $\Db(X_1)\to\Db(X_2)$ induced by exact functors $\Coh(X_1)\to\Coh(X_2)$. Here we assume that $X_1$ and $X_2$ are smooth projective varieties and that $\dim(X_1)>0$.
\end{ex}

As a consequence of Proposition \ref{prop:smooth}, Theorem
\ref{thm:main2} generalizes the main result of \cite{CS} when
the twists from the Brauer groups are trivial.

\section{Enhancements and existence of Fourier--Mukai kernels}\label{sec:existence}

In this section we show how to construct Fourier--Mukai kernels for
functors satisfying the condition $\Cc$ defined in the introduction. This extends several results already present in the literature. Moreover we show that, in the supported setting, the Fourier--Mukai kernels have to be quasi-coherent rather than coherent. We need also to recall some basic facts about dg categories. As an application of this machinery and of the results in the previous sections, we get the proof of Theorem \ref{thm:main1}. Here and for the rest of the paper, we fix a universe such that all dg categories are small dg categories with respect to this universe (see \cite[Appendix A]{LO}).

\subsection{Dg categories}\label{subsec:dg}

In this section we give a quick introduction to some basic definitions and results about dg categories and dg functors. For a survey on the subject, the reader can have a look at \cite{K}.

Recall that a \emph{dg category} is a $\K$-linear category $\cat{A}$
such that, for all $A,B\in\Ob(\cat{A})$, the morphism spaces
$\Hom(A,B)$ are $\ZZ$-graded $\K$-modules with a differential
$d\colon\Hom(A,B)\to\Hom(A,B)$ of degree $1$ and the composition maps
are morphisms of complexes. Notice that the identity of each object is
a closed morphism of degree $0$.

\begin{ex}\label{ex:dgcat1}
(i) Any $\K$-linear category has a (trivial) structure of dg category,
with morphism spaces concentrated in degree $0$.

(ii) For a dg category $\cat{A}$, one defines the opposite dg category $\cat{A}\opp$ with $\Ob(\cat{A}\opp)=\Ob(\cat{A})$ while $\Hom_{\cat{A}\opp}(A,B):=\Hom_{\cat{A}}(B,A)$.

(iii) Following \cite{Dr}, given a dg category $\cat{A}$ and a full dg
subcategory $\cat{B}$ of $\cat{A}$, one can form the quotient
$\cat{A}/\cat{B}$ which is again a dg category.

(iv) Given an abelian category $\cat{A}$, one can consider the dg
category $\Cdg(\cat{A})$ of complexes of objects in $\cat{A}$, its
full dg subcategory $\Acdg(\cat{A})$ of acyclic complexes and the
dg quotient $\Dg(\cat{A}):=\Cdg(\cat{A})/\Acdg(\cat{A})$. When
$\cat{A}=\Qcoh_Z(X)$, for $X$ a scheme containing a closed subscheme $Z$, we denote the dg categories $\Cdg(\cat{A})$,
$\Acdg(\cat{A})$ and $\Dg(\cat{A})$ respectively by $\Cdg_Z(X)$,
$\Acdg_Z(X)$ and $\Dg_Z(X)$.
\end{ex}

Given a dg category $\cat{A}$ we denote by $H^0(\cat{A})$ its
\emph{homotopy} category. The objects of $H^0(\cat{A})$ are the same
as those of $\cat{A}$ while the morphisms from $A$ to $B$ are obtained
by taking the $0$-th cohomology $H^0(\Hom_{\cat{A}}(A,B))$ of the
complex $\Hom_{\cat{A}}(A,B)$. If $\cat{A}$ is pretriangulated (see
\cite{K} for the definition), then $H^0(\cat{A})$ has a natural
structure of triangulated category.

\begin{ex}\label{ex:dgcat2}
Given an abelian category $\cat{A}$, the dg categories $\Cdg(\cat{A})$,
$\Acdg(\cat{A})$ and $\Dg(\cat{A})$ are pretriangulated and, as it is
explained for example in \cite[Sect.\ 4.4]{K}, there is an exact
equivalence between the derived category $\D(\cat{A})$ and the
homotopy category $H^0(\Dg(\cat{A}))$. In particular, when $X$ is a
scheme containing a closed subscheme $Z$, the dg category
$\Dg_Z(X)$ is pretriangulated and $H^0(\Dg_Z(X))\iso\D_Z(\Qcoh(X))$
(here one uses that $\D_Z(\Qcoh(X))$ is naturally equivalent to
$\D(\Qcoh_Z(X))$ by Proposition \ref{prop:Ball2}).
\end{ex}

\medskip

A \emph{dg functor} $\fun{F}\colon\cat{A}\to\cat{B}$ between two dg
categories is the datum of a map $\Ob(\cat{A})\to\Ob(\cat{B})$ and of
morphisms of complexes of $\K$-modules
$\Hom_{\cat{A}}(A,B)\to\Hom_{\cat{B}}(\fun{F}(A),\fun{F}(B))$, for
$A,B\in\Ob(\cat{A})$, which are compatible with the compositions and
the units.

A dg functor $\fun{F}\colon\cat{A}\to\cat{B}$ induces a functor
$H^0(\fun{F})\colon H^0(\cat{A})\to H^0(\cat{B})$, which is exact
(between triangulated categories) if $\cat{A}$ and $\cat{B}$ are
pretriangulated.

A dg functor $\fun{F}\colon\cat{A}\to\cat{B}$ is a
\emph{quasi-equivalence}, if the maps
$\Hom(A,B)\to\Hom(\fun{F}(A),\fun{F}(B))$ are quasi-isomorphisms, for
every $A,B\in\cat{A}$, and $H^0(\fun{F})$ is an equivalence. One can
consider the localization $\Hqe$ of the category of dg categories over
$\K$ with respect to quasi-equivalences (\cite{To}). Given
a dg functor $\fun{F}$, we will denote with the
same symbol its image in $\Hqe$. In particular, if $\fun{F}$ is a
quasi-equivalence, we denote by $\fun{F}^{-1}$ the morphism in $\Hqe$ which is the inverse of $\fun{F}$.

For a small dg category $\cat{A}$, one can consider the
pretriangulated dg category $\dgMod{\cat{A}}$ of \emph{right dg
$\cat{A}$-modules}. A right dg $\cat{A}$-module is a dg functor
$\fun{M}\colon\cat{A}\opp\to\dgMod{\K}$, where $\dgMod{\K}$ is the dg
category of dg $\K$-modules. The full dg subcategory of acyclic right
dg modules is denoted by $\Ac(\cat{A})$, and $H^0(\Ac(\cat{A}))$ is a
full triangulated subcategory of the homotopy category
$H^0(\dgMod{\cat{A}})$. Hence the \emph{derived category} of the dg
category $\cat{A}$ is the Verdier quotient
\[
\dgD(\cat{A}):=H^0(\dgMod{\cat{A}})/H^0(\Ac(\cat{A})).
\]
A right dg $\cat{A}$-module is \emph{representable} if it is contained
in the image of the Yoneda dg functor
\[
\dgYon\colon\cat{A}\to\dgMod{\cat{A}}\qquad
A\mapsto\Hom_{\cat{A}}(-,A)=:\dgYon[A].
\]
A right dg $\cat{A}$-module is \emph{free} if it is isomorphic to a
direct sum of dg modules of the form
$\dgYon[A][m]$, where $A\in\cat{A}$ and $m\in\ZZ$. A right dg $\cat{A}$-module $\fun{M}$ is \emph{semi-free} if it has a filtration
\[
0=\fun{M}_0\subseteq\fun{M}_1\subseteq\ldots=\fun{M}
\]
such that $\fun{M}_i/\fun{M}_{i-1}$ is free, for all $i$. We denote by
$\SF(\cat{A})$ the full dg subcategory of semi-free dg modules, while
$\SFfg(\cat{A})\subseteq\SF(\cat{A})$ is the full dg subcategory of
\emph{finitely generated semi-free} dg modules. Namely, there is $n$
such that $\fun{M}_n=\fun{M}$ and each $\fun{M}_i/\fun{M}_{i-1}$ is a
finite direct sum of dg modules of the form $\dgYon[A][m]$. The dg
modules which are homotopy equivalent to direct summands of finitely
generated semi-free dg modules are called \emph{perfect} and they form
a full dg subcategory $\Perf(\cat{A})$.

Following \cite{K,To}, given two dg categories $\cat{A}$ and $\cat{B}$, we denote by $\rep(\cat{A},\cat{B})$ the full subcategory of the derived category
$\dgD(\cat{A}\opp\otimes_\K\cat{B})$ of $\cat{A}$-$\cat{B}$-bimodules $\fun{C}$ such that the functor
$(-)\otimes_{\cat{A}}\fun{C}\colon\dgD(\cat{A})\to\dgD(\cat{B})$
sends the representable $\cat{A}$-modules to objects which are
isomorphic to representable $\cat{B}$-modules. An object in
$\rep(\cat{A},\cat{B})$ is called a \emph{quasi-functor}.
By \cite{To}, morphisms in $\Hqe$ are in natural bijection with
isomorphism classes of quasi-functors. Thus, with a slight abuse of notation, we sometimes call quasi-functor a morphism in $\Hqe$. Notice that a quasi-functor
$\fun{M}\in\rep(\cat{A},\cat{B})$ induces a functor
$H^0(\fun{M})\colon H^0(\cat{A})\to H^0(\cat{B})$, well defined up to
isomorphism.

For $\fun{F}\colon\cat{A}\to\cat{B}$ a dg functor, there exist dg functors
\[
\fun{F}^*\colon\dgMod{\cat{A}}\to\dgMod{\cat{B}}\qquad
\fun{F}_*\colon\dgMod{\cat{B}}\to\dgMod{\cat{A}}
\]
also denoted, respectively, by $\Ind_{\fun{F}}$ and $\Res_{\fun{F}}$.
While $\fun{F}_*$ is simply induced by composition with $\fun{F}$, the
reader can have a look at \cite[Sect.\ 14]{Dr} for the definition and
properties of $\fun{F}^*$. In particular, $\fun{F}^*$ is left adjoint
to $\fun{F}_*$ and commutes with the Yoneda embeddings, up to dg
isomorphism. Moreover, $\fun{F}^*$ preserves semi-free dg modules and
$\fun{F}^*\colon\SF(\cat{A})\to\SF(\cat{B})$ is a quasi-equivalence if
$\fun{F}\colon\cat{A}\to\cat{B}$ is such.

\medskip

Given two pretriangulated dg categories $\cat{A}$ and $\cat{B}$ and
an exact functor $\fun{F}\colon H^0(\cat{A})\to H^0(\cat{B})$, a \emph{dg lift} of $\fun{F}$ is a quasi-functor $\fun{G}\in\rep(\cat{A},\cat{B})$ such that $H^0(\fun{G})\iso\fun{F}$.

An \emph{enhancement} of a triangulated category $\cat{T}$ is a pair
$(\cat{A},\alpha)$, where $\cat{A}$ is a pretriangulated
dg category and $\alpha\colon H^0(\cat{A})\to\cat{T}$ is an exact
equivalence. The enhancement $(\cat{A},\alpha)$ of $\cat{T}$ is
\emph{unique} if for any enhancement $(\cat{B},\beta)$ of $\cat{T}$
there exists a quasi-functor $\gamma\colon\cat{A}\to\cat{B}$ such that
$H^0(\gamma)\colon H^0(\cat{A})\to H^0(\cat{B})$ is an exact equivalence. We
say that the enhancement is \emph{strongly unique} if moreover
$\gamma$ can be chosen so that
$\alpha\iso\beta\comp H^0(\gamma)$. Often, by abuse of notation, we will say that a pretriangulated dg category $\cat{A}$ is an enhancement of a triangulated category $\cat{T}$ when there is an exact equivalence $H^0(\cat{A})\iso\cat{T}$.

\begin{ex}\label{ex:derenh}
If $\cat{A}$ is a dg category, $\SF(\cat{A})$ and $\Perf(\cat{A})$ are
enhancements, respectively, of $\dgD(\cat{A})$ and
$\dgD(\cat{A})^c$.
\end{ex}

\begin{ex}\label{ex:enh}
Let $\cat{A}$ be a pretriangulated dg category and $\cat{B}$ a full pretriangulated dg subcategory of $\cat{A}$. By \cite{Dr}, there exists a natural exact equivalence between the Verdier quotient $H^0(\cat{A})/H^0(\cat{B})$ and $H^0(\cat{A}/\cat{B})$. Hence $\cat{A}/\cat{B}$ is an enhancement of $H^0(\cat{A})/H^0(\cat{B})$.
\end{ex}

\begin{ex}\label{ex:dgcat3}
By Example \ref{ex:dgcat2}, $\Dg(\cat{A})$ (for $\cat{A}$ an abelian
category) is an enhancement of $\D(\cat{A})$. Moreover, if $X$ is a
scheme containing a closed subscheme $Z$, $\Dg_Z(X)$ is an
enhancement of $\D_Z(\Qcoh(X))$. Let $\Perf_{Z}(X)$ be the full dg subcategory of $\Dg_{Z}(X)$ consisting of compact objects in $H^0(\Dg_{Z}(X))$. Notice that $\Perf_{Z}(X)$ is an enhancement of $\Dp[Z](X)$ and, as we mentioned above, we will identify $H^0(\Perf_{Z}(X))$ with $\Dp[Z](X)$, to make the notation simpler.
\end{ex}

\subsection{Enhancements and the proof of Theorem \ref{thm:main1}}\label{subsec:enh}

Let $X$ be a quasi-projective scheme containing a projective
subscheme $Z$ and let $H$ be an ample divisor on $X$. Assume that
$\ko_{iZ}\in\Dp(X)$ for all $i>0$ (hence the
full subcategory $\cat{Amp}(Z,X,H)$ defined in
\eqref{eqn:almamp} is contained in $\Dp[Z](X)$). Consider the
dg category $\cat{A}:=\Coh_Z(X)\cap\Dp[Z](X)$ concentrated in degree
zero and notice that, due to Remark \ref{rmk:ampleset} and Proposition
\ref{prop:amp}, the objects in $\cat{A}$ satisfy \eqref{ample2} in
Definition \ref{def:almostample}. By abuse of notation, we will write
$\dYon$ for the functor $\cat{A}\to\dgD(\cat{A})$ which is the
composition of $H^0(\dgYon)\colon\cat{A}=H^0(\cat{A})\to
H^0(\dgMod{\cat{A}})$ and of the quotient functor
$H^0(\dgMod{\cat{A}})\to\dgD(\cat{A})$.
As a matter of notation, if $\cat{T}$ is a triangulated category with
arbitrary direct sums and $\cat{L}$ is a localizing subcategory of
$\cat{T}$, we denote by $\pi\colon\cat{T}\to\cat{T}/\cat{L}$ the quotient functor. Recall that a full triangulated
subcategory $\cat{S}$ of a triangulated category $\cat{T}$ is
\emph{localizing} if it is closed under arbitrary direct
sums.

\begin{lem}\label{lem:uniqueenhance}
There exists an exact equivalence
$\varphi\colon\D_Z(\Qcoh(X))\to\dgD(\cat{A})/\cat{L}$, for some
localizing subcategory $\cat{L}\subseteq\dgD(\cat{A})$, such that we have an isomorphism of functors $\pi\comp\dYon\iso\varphi\rest{\cat{A}}$. Moreover $\D_Z(\Qcoh(X))$ has a unique enhancement.
\end{lem}

\begin{proof}
By Proposition \ref{prop:amp}, the category
$\cat{Amp}(Z,X,H)\subseteq\cat{A}$ is a set of compact generators for
the Grothendieck category $\Qcoh_{Z}(X)$.
Then take the abelian category $\Mod{\cat{A}}$ of modules over
$\cat{A}$, i.e.\ $\K$-linear contravariant functors from $\cat{A}$ to
the category of $\K$-modules. As it is explained in \cite{CS4} there is a Serre subcategory $\cat{N}$ of $\Mod{\cat{A}}$ such that $\Qcoh_{Z}(X)\iso\Mod{\cat{A}}/\cat{N}$. By \cite[Lemma 7.2]{LO}, we then have an equivalence
\begin{equation}\label{eqn:inet}
\D(\Qcoh_{Z}(X))\iso\D(\Mod{\cat{A}})/\D_{\cat{N}}(\Mod{\cat{A}}),
\end{equation}
where $\D_{\cat{N}}(\Mod{\cat{A}})$ is the full subcategory of
$\D(\Mod{\cat{A}})$ consisting of complexes with cohomologies in
$\cat{N}$. As observed at the beginning of Section 7 in \cite{LO},
there exists a natural equivalence
$\psi\colon\D(\Mod{\cat{A}})\to\dgD(\cat{A})$. Hence we set $\cat{L}$ to be the
full subcategory of $\dgD(\cat{A})$ corresponding to
$\D_{\cat{N}}(\Mod{\cat{A}})$ under $\psi$.

We define $\varphi$ to be the composition of \eqref{eqn:inet} with the
equivalences
$\D(\Mod{\cat{A}})/\D_{\cat{N}}(\Mod{\cat{A}})\iso\dgD(\cat{A})/\cat{L}$
(induced by $\psi$) and $\D_{Z}(\Qcoh(X))\iso\D(\Qcoh_{Z}(X))$ (see
Proposition \ref{prop:Ball2}). The fact that there is an isomorphism
of functors $\pi\comp\dYon\iso\varphi\rest{A}$ is observed in
\cite{CS4}, where the above construction is analyzed further. Notice
that the Yoneda embedding $\dYon\colon\cat{A}\to\D(\cat{A})$ coincides
with the classical Yoneda embedding $\cat{A}\to\Mod{\cat{A}}$, composed
with the natural inclusion $\Mod{\cat{A}}\mono\D(\Mod{\cat{A}})$ and
with $\psi$.

The second part of the statement is a straightforward consequence of \cite[Thm.\ 7.5]{LO}.
\end{proof}

As a consequence, we have an equivalence
\begin{equation}\label{eqn:alpha}
\alpha\colon\Dp[Z](X)\to(\dgD(\cat{A})/\cat{L})^c
\end{equation}
induced by
$\varphi$ (i.e.\ we set $\alpha:=\varphi\rest{\Dp[Z](X)}$). This is
because $\varphi$, being an equivalence, sends compact objects to
compact objects. By Lemma \ref{lem:uniqueenhance}, we have an isomorphism of functors
$\alpha^{-1}\comp\pi\comp\dYon\iso\id_{\cat{A}}$ (where $\alpha^{-1}$ denotes a
quasi-inverse of $\alpha$).

Let $\cat{L}'$ be the lift of the localizing subcategory $\cat{L}$ to $\dgMod{\cat{A}}$. Let $\cat{D}$ be the full dg subcategory of
$\SF(\cat{A})/(\SF(\cat{A})\cap\cat{L}')$ (which, by Examples
\ref{ex:derenh} and \ref{ex:enh}, is an enhancement of
$\dgD(\cat{A})/\cat{L}$) consisting of the compact objects in $H^0(\SF(\cat{A})/(\SF(\cat{A})\cap\cat{L}'))$. Obviously, $\cat{D}$ is an enhancement of $(\dgD(\cat{A})/\cat{L})^c$ in a natural way. In view of this, we will identify $H^0(\cat{D})$ with $(\dgD(\cat{A})/\cat{L})^c$.

\begin{lem}\label{lem:lift}
If $(\cat{B},\beta)$ is an enhancement of $\Dp[Z](X)$, there exists a quasi-functor $\delta\colon\cat{D}\to\cat{B}$ such that $H^0(\delta)$ is an exact equivalence and there is an isomorphism of functors $\cat{A}\to H^0(\cat{B})$
\begin{equation}\label{eqn:isouniq1}
H^0(\delta)\comp\pi\comp\dYon\isomor(\alpha\comp\beta)^{-1}\comp\pi\comp\dYon.
\end{equation}
\end{lem}

\begin{proof}
Due to \cite{CS4}, the category $\cat{L}$ is compactly generated in $\dgD(\cat{A})\iso\D(\Mod{\cat{A}})$ and $\cat{L}^c=\cat{L}\cap\dgD(\cat{A})^c$. Hence we can apply \cite[Thm.\ 6.4]{LO} to the exact equivalence $(\alpha\comp\beta)^{-1}$, providing the quasi-functor $\delta\colon\cat{D}\to\cat{B}$ in the statement. Indeed, $H^0(\delta)$ is fully faithful (by (1) in \cite[Thm.\ 6.4]{LO}), satisfies \eqref{eqn:isouniq1} (by (2) in \cite[Thm.\ 6.4]{LO}) and is essentially surjective (by (3) in \cite[Thm.\ 6.4]{LO} and the fact that $(\alpha\comp\beta)^{-1}$ is an equivalence).
\end{proof}

Now we want to prove Theorem \ref{thm:main1} and so we assume further
that $T_0(\ko_{Z})=0$. Let us reproduce the statement here for the convenience of the reader.

\begin{thm}
Let $X$ be a quasi-projective scheme containing a projective
subscheme $Z$ such that $\ko_{iZ}\in\Dp(X)$, for all $i>0$, and
$T_0(\ko_{Z})=0$. Then $\Dp[Z](X)$ has a strongly unique enhancement.
\end{thm}

\begin{proof}
Let $(\cat{B},\beta)$ be an enhancement of $\Dp[Z](X)$. By Lemma \ref{lem:lift}, there exists a quasi-functor $\delta\colon\cat{D}\to\cat{B}$ such that $H^0(\delta)$ is an exact equivalence satisfying \eqref{eqn:isouniq1}.

By Proposition \ref{prop:excat}, there is a natural exact equivalence
$\varepsilon\colon\Db(\cat{E})\to\Dp[Z](X)$, for the exact category
$\cat{E}=\Dp[Z](X)\cap\Coh_Z(X)$. Setting
$\varpi:=\alpha\comp\varepsilon$, we have
$\varpi^{-1}\comp\pi\comp\dYon\iso\id_{\cat{A}}$.

Put $\fun{F}_1:=H^0(\delta)\comp\varpi$ and $\fun{F}_2:=(\alpha\comp\beta)^{-1}\comp\varpi$. Then \eqref{eqn:isouniq1} reads as
\begin{equation}\label{eqn:isoimp}
\fun{F}_1\rest{\cat{A}}\isomor\fun{F}_2\rest{\cat{A}}.
\end{equation}
By Corollary \ref{cor:critiso2}, it extends to a unique isomorphism
$\fun{F}_1\iso\fun{F}_2$. Notice that this is the point where we use
that $T_0(\ko_Z)=0$ as, under this assumption, $\cat{A}$ is an almost ample set (by Proposition \ref{prop:amp} and Remark \ref{rmk:ampleset}) and, using this, every full functor certainly satisfies $\Ca$ (see Remark \ref{rmK:2cond} and Example \ref{ex:condfun}). Therefore there exists an isomorphism between $H^0(\delta)$ and $(\alpha\comp\beta)^{-1}$.

This proves that the enhancement of $\Dp[Z](X)$ is strongly
unique. Indeed, suppose that $(\cat{B}_1,\beta_1)$ and
$(\cat{B}_2,\beta_2)$ are enhancements of
$(\dgD(\cat{A})/\cat{L})^c$. By the above discussion, there are
quasi-equivalences $\delta_i\colon\cat{D}\to\cat{B}_i$ and unique
isomorphisms $H^0(\delta_i)\iso\beta_i^{-1}\comp\alpha^{-1}$. To
conclude, if we set
$\tilde\delta:=\delta_2\comp\delta_1^{-1}\colon\cat{B}_1\to\cat{B}_2$,
we have $\beta_2\comp H^0(\tilde\delta)\iso\beta_1$.
\end{proof}

In view of Example \ref{ex:geosetting}, it is straightforward to deduce the following special instance of Theorem \ref{thm:main1}.

\begin{cor}\label{cor:LO1}
Let $X$ be a quasi-projective scheme containing a projective
subscheme $Z$ such that $T_0(\ko_{Z})=0$ and either $X$ is smooth or $X=Z$. Then $\Dp[Z](X)$ has a strongly unique enhancement.
\end{cor}

If $X=Z$, then this is nothing but one of the main results in \cite{LO} (see Theorem 9.9 there).

\subsection{The Fourier--Mukai kernels}\label{subsec:kernel1}

The aim of this subsection is to prove the following result, which is the first part of Theorem \ref{thm:main2}.

\begin{prop}\label{prop:rep1}
Let $X_1$ be a quasi-projective scheme containing a projective
subscheme $Z_1$ such that $\ko_{iZ_1}\in\Dp(X_1)$, for all $i>0$. Assume that $X_2$ is a scheme containing a
closed subscheme $Z_2$.
Then, for any exact functor $\fun{F}\colon\Dp[Z_1](X_1)\to\Dp[Z_2](X_2)$
satisfying $\Cc$ there exist $\ke\in\D_{Z_1\times Z_2}(\Qcoh(X_1\times
X_2))$ and an isomorphism of exact functors $\fun{F}\iso\FMS{\ke}$.
\end{prop}

To construct the Fourier--Mukai kernel realizing $\fun{F}$ as a Fourier--Mukai functor we will make use of some ideas from Sections 4, 6 and 9 of \cite{LO}. Clearly there is no space for reproducing the arguments in \cite{LO} in full details. Thus, for the convenience of the reader, we outline here the main steps in the proof of Proposition \ref{prop:rep1}.

The first passage consists in the construction of a quasi-functor
$\dgfun{F}_1$ which will turn out to be a dg lift of $\fun{F}$. This
is done in Section \ref{subsubsec:dglift} following
\cite[Sect.\ 4]{LO}. Then we will show, in Section
\ref{subsubsec:iso}, the existence of the above mentioned
isomorphism. The first main ingredient at this point is that the dg
enhancement of the category of perfect complexes with support
conditions is constructed out of a set satisfying \eqref{ample2} of Definition \ref{def:almostample}. The second
ingredient consists in showing that $\fun{F}$ and
$\fun{F}_1:=H^0(\dgfun{F}_1)$ are isomorphic on such a
set. For this we need to apply some tricks from
\cite[Sect.\ 6]{LO}. Hence, using that the functor $\fun{F}$ satisfies $\Cc$, the results in Section \ref{sec:extending}
apply, proving that $\fun{F}\iso\fun{F}_1$. Finally, we prove in
Section \ref{subsubsec:kernel} that $\fun{F}_1$ is of Fourier--Mukai
type. This essentially follows, as in \cite[Sect.\ 9]{LO}, from an
application of \cite{To}.

We will discuss the relevant points where one can avoid the fully faithful assumption on $\fun{F}$ which is present in \cite{LO}. Anyway, it is important to observe that the steps in the following proof where such an assumption may be relevant are where we make use of Section 6 in \cite{LO}. On the other hand, only two results from that part of \cite{LO} are used here: Lemmas 6.1 and 6.2. While the first one is absolutely innocuous, one needs to observe that $\Cc$ (together with \cite{CS4}) is enough to apply \cite[Lemma 6.2]{LO}.

\subsubsection{The dg lift}\label{subsubsec:dglift}

We keep the same notation as in Section \ref{subsec:enh}, with $X_1$
and $Z_1$ playing the same roles as $X$ and $Z$ there. Set
$\fun{F}':=\fun{F}\comp\alpha^{-1}$, where $\alpha$ is defined in
\eqref{eqn:alpha}. Since the essential image of $\pi\comp\dYon$ is
contained in $(\dgD(\cat{A})/\cat{L})^c$, we can define the full dg
subcategory $\cat{B}:=\{\fun{F}'(\pi(\dYon[A])):A\in\cat{A}\}$ of
$\Perf_{Z_2}(X_2)$ (see Example \ref{ex:dgcat3}).

Let $\cat{C}$ be the full dg subcategory of $\Perf_{Z_2}(X_2)$ such that $H^0(\cat{C})$ is classically generated by the objects in $\cat{B}$. By definition the functor $\fun{F}'\colon(\dgD(\cat{A})/\cat{L})^c\to\Dp[Z_2](X_2)$ factors through $H^0(\cat{C})$ giving rise to the functor
\[
\cat{A}\mor{\pi\comp\dYon}(\dgD(\cat{A})/\cat{L})^c\mor{\fun{F}'} H^0(\cat{C})\mono\Dp[Z_2](X_2)
\]
which, in turn, factors through a functor $\newrho_1\colon\cat{A}\to H^0(\cat{B})$, which can be viewed as a dg functor in a trivial way.

Consider the dg category $\tau_{\le0}\cat{B}$ with the same objects as
$\cat{B}$ but such that $\Hom_{\tau_{\le0}\cat{B}}(\ke,\kf)=
\tau_{\le0}\Hom_{\cat{B}}(\ke,\kf)$ (here $\tau_{\le0}$ is the gentle
truncation). Let $p\colon\tau_{\le0}\cat{B}\to H^0(\cat{B})$ and
$l\colon\tau_{\le0}\cat{B}\to\cat{B}$ be the natural dg functors. Due
to Lemma \ref{lem:uniqueenhance} and part (1) of $\Cc$, $p$ is a quasi-equivalence. Thus, we get the quasi-functor
\[
\newrho_2\colon\SF(\cat{A})\mor{}\SF(H^0(\cat{B}))\to\SF(\tau_{\leq 0}\cat{B})\mor{}\SF(\cat{B}),
\]
where $\newrho_2=l^*\comp(p^*)^{-1}\comp\newrho_1^*$ and $(p^*)^{-1}$ is the inverse in $\Hqe$ of $p^*$.

Since $\cat{L}=H^0(\cat{L}')$ is compactly generated by \cite{CS4}, the argument in \cite[Lemma 6.2]{LO} applies and the quasi-functor $\newrho_2$ factors through the dg quotient $\SF(\cat{A})/(\SF(\cat{A})\cap\cat{L}')$. Hence we get a quasi-functor
\[
\newrho_3\colon\SF(\cat{A})/(\SF(\cat{A})\cap\cat{L}')\mor{}\SF(\cat{B}).
\]
It is important to stress here that, in the proof of \cite[Lemma
  6.2]{LO}, the assumption that $\fun{F}$ is fully faithful is not
needed and $\Cc$ suffices to conclude. This is because the proof
relies on \cite[Prop.\ 3.4]{LO}, where part (1) of $\Cc$ is enough.

The quasi-functor $\newrho_3$ restricts to a quasi-functor $\newrho_3\colon\cat{D}\to\Perf(\cat{B})$. Indeed, it is enough to observe that by definition $H^0(\newrho_3)(\pi(\dYon[A]))\in H^0(\Perf(\cat{B}))$, for all $A\in\cat{A}$. Now, \cite[Prop.\ 1.16]{LO} applies and gives a quasi-equivalence $\psi\colon\cat{C}\to\Perf(\cat{B})$. Take the quasi-functor
\begin{equation}\label{eqn:LO1}
\newrho_4\colon\cat{D}\mor{\newrho_3}\Perf(\cat{B})\mor{\psi^{-1}}\cat{C}\mono\Perf_{Z_2}(X_2),
\end{equation}
where, as usual, $\psi^{-1}$ is the inverse in $\Hqe$ of the
quasi-equivalence $\psi$.

We apply Lemma \ref{lem:lift} to the enhancement $\Perf_{Z_1}(X_1)$ of $\Dp[Z_1](X_1)$ getting a quasi-functor $\delta\colon\cat{D}\to\Perf_{Z_1}(X_1)$,
whence
\[
\dgfun{F}_1:=\newrho_4\comp\delta^{-1}\colon\Perf_{Z_1}(X_1)\lto\Perf_{Z_2}(X_2).
\]

\subsubsection{The isomorphism}\label{subsubsec:iso}

Consider the exact functor
\[
H^0(\newrho_4)\colon(\dgD(\cat{A})/\cat{L})^c(=H^0(\cat{D}))\lto\Dp[Z_2](X_2).
\]
As a consequence of \cite[Lemma 6.1]{LO} (see also \cite[Prop.\ 3.4]{LO}) and of the definition in \eqref{eqn:LO1}, we get an isomorphism
\[
\theta\colon\fun{F}'\comp\pi\comp\dYon\isomor H^0(\newrho_4)\comp\pi\comp\dYon,
\]
as functors from $\cat{A}$ to $\Dp[Z_2](X_2)$. Again, it is important to observe that we do not need to have $\fun{F}$ (and thus $\fun{F}'$) fully faithful to apply \cite[Lemma 6.1]{LO} being the isomorphism above a simple consequence of the definition of $\newrho_4$.

From the isomorphisms \eqref{eqn:isouniq1} and $\alpha^{-1}\comp\pi\comp\dYon\iso \id_{\cat{A}}$ (see Lemma \ref{lem:uniqueenhance}), $\theta$ gives an isomorphism
\[
\fun{F}\rest{\cat{A}}\isomor\fun{F}_1\rest{\cat{A}},
\]
where $\fun{F}_1:=H^0(\dgfun{F}_1)$. Applying Corollary \ref{cor:critiso2} we get an isomorphism of exact functors
\begin{equation}\label{eqn:LO2}
\fun{F}\isomor\fun{F}_1.
\end{equation}

\subsubsection{The kernel}\label{subsubsec:kernel}

By \cite[Prop.\ 1.17]{LO} (which can be applied here in view of
Proposition \ref{prop:catcomp}), we get a quasi-equivalence $\varphi_i\colon\SF(\Perf_{Z_i}(X_i))\isomor\Dg_{Z_i}(X_i)$. Hence, composing the extension of $\dgfun{F}_1$ to semi-free modules with $\varphi_i$, we get a quasi-functor
\[
\dgfun{F}_2\colon\Dg_{Z_1}(X_1)\mor{\varphi_1^{-1}}\SF(\Perf_{Z_1}(X_1))\to\SF(\Perf_{Z_2}(X_2))\mor{\varphi_2}\Dg_{Z_2}(X_2)
\]
whose $H^0$ commutes with direct sums (because it has a right adjoint,
according to \cite[Sect.\ 1]{LO}). Observe that
\begin{equation}\label{eqn:Fi}
\fun{F}_1\iso H^0(\dgfun{F}_2)\rest{\Dp[Z_1](X_1)}.
\end{equation}

The easier case $X_i=Z_i$, for $i=1,2$, generalizing
\cite[Cor.\ 9.13]{LO} (see, in particular, parts (2) and (3) there)
can be treated already.

\begin{cor}\label{cor:experf}
Let $X_1$ be a projective variety and let $X_2$ be a scheme. For any
exact functor $\fun{F}\colon\Dp(X_1)\to\Dp(X_2)$ satisfying $\Cc$,
there exist $\ke\in\Db(X_1\times X_2)$ and an isomorphism of exact
functors $\fun{F}\iso\FM{\ke}$.
\end{cor}

\begin{proof}
By \cite[Thm.\ 8.9]{To} there is $\ke\in\D(\Qcoh(X_1\times X_2))$ such that $\FM{\ke}\iso H^0(\dgfun{F}_2)$.
As $\fun{F}_1\iso H^0(\dgfun{F}_1)\iso H^0(\dgfun{F}_2)\rest{\Dp[](X_1)}$, the isomorphism \eqref{eqn:LO2} gives $\fun{F}\iso\FM{\ke}$. The fact that $\ke$ is bounded coherent is obtained by the same argument as in the proof of \cite[Cor.\ 9.13]{LO}, part (4). We do not explain this here as this is a special instance of Lemma \ref{lem:bound}.
\end{proof}

Back to the general setting, we can observe the following.

\begin{remark}
Assume that $X$ is a quasi-projective scheme containing a projective subscheme $Z$ such that $\ko_{iZ}\in\Dp(X)$, for all $i>0$. The exact functors $\iota\colon\D_Z(\Qcoh(X))\lto\D(\Qcoh(X))$ and
$\iota^!\colon\D(\Qcoh(X))\to\D_Z(\Qcoh(X))$, defined in Section
\ref{subsec:prelcat}, have natural dg lifts (denoted with the same
symbols) $\iota\colon\Dg_Z(X)\lto\Dg(X)$ and
$\iota^!\colon\Dg(X)\to\Dg_Z(X)$.

To show this one can use \cite[Sect.\ 4]{KL}. More precisely, the presence of the right adjoint $\iota^!$ yields a semi-orthogonal decomposition of $\D(\Qcoh(X))$, where $\D_{Z}(\Qcoh(X))$ is a non-trivial piece. By \cite[Prop.\ 4.10]{KL}, such a decomposition can be written at the dg level as a gluing with one piece given by $\Dg_Z(X)$, up to quasi-equivalence.
Hence \cite[Lemma 4.4]{KL} combined with the uniqueness of the
enhancements of $\D(\Qcoh(X))$ (see \cite[Cor.\ 2.11]{LO}) and of
$\D_{Z}(\Qcoh(X))$ (see Lemma \ref{lem:uniqueenhance}) provide the dg lifts for the functors $\iota$ and $\iota^!$. By the construction, it is clear that the dg lifts of $\iota$ and $\iota^!$ are not, in general, dg functors but just quasi-functors.
\end{remark}

Under the light of the above remark, consider the quasi-functor $\dgfun{F}_3$ making the following diagram commutative
\begin{equation}\label{eqn:comdiaLO}
\xymatrix{
\Dg(X_1)\ar[rr]^-{\dgfun{F}_3}\ar[d]_{\iota_1^!}&&\Dg(X_2)\\
\Dg_{Z_1}(X_1)\ar[rr]^-{\dgfun{F}_2}&&\Dg_{Z_2}(X_2).\ar[u]_{\iota_2}
}
\end{equation}
Clearly $H^0(\dgfun{F}_3)$ commutes with direct sums, as the same is
true for $\iota_1^!$ (by Lemma \ref{lem:sums}), $H^0(\dgfun{F}_2)$ and
$\iota_2$. Notice also that, if $Z_1=X_1$, then $\iota_1^!\iso\id$.

By \cite[Thm.\ 8.9]{To}, there exist $\widetilde{\ke}\in\D(\Qcoh(X_1\times X_2))$ and an isomorphism of exact functors
\[
\FM{\widetilde{\ke}}\iso H^0(\dgfun{F}_3)\colon\D(\Qcoh(X_1))\mor{\iota_1^!}\D_{Z_1}(\Qcoh(X_1))\mor{H^0(\dgfun{F}_2)}\D_{Z_2}(\Qcoh(X_2))\mor{\iota_2}\D(\Qcoh(X_2)).
\]
It follows that, setting $\ke:=\iota_{1,2}^!(\widetilde{\ke})\in
\D_{Z_1\times Z_2}(\Qcoh(X_1\times X_2))$, we have
\[
H^0(\dgfun{F}_2)\iso
\iota_2^!\comp\iota_2\comp H^0(\dgfun{F}_2)\comp\iota_1^!\comp\iota_1
\iso\iota_2^!\comp\FM{\widetilde{\ke}}\comp\iota_1\iso\FMS{\ke}
\]
by Lemma \ref{lem:FMS}.

\begin{remark}\label{rmk:Toenrevis}
	The careful reader may notice that To\"{e}n's result \cite[Thm.\ 8.9]{To} was originally proved by using different enhancements for $\D(\Qcoh(X_1))$ and $\D(\Qcoh(X_2))$. That the same statement holds true for $\Dg(X_1)$ and $\Dg(X_2)$ is observed in \cite[Thm.\ 4.9]{K}.
\end{remark}

Putting this together with the isomorphisms \eqref{eqn:LO2} and
\eqref{eqn:Fi}, we have proved Proposition \ref{prop:rep1}. By Example
\ref{ex:geosetting}, the following consequence of Proposition
\ref{prop:rep1} is immediate.

\begin{cor}\label{cor:LO2}
Let $X_1$ be a quasi-projective scheme containing a projective subscheme $Z_1$ such that either $X_1$ is smooth or $X_1=Z_1$. Assume that $X_2$ is a scheme containing a closed subscheme $Z_2$.
Then, for any exact functor $\fun{F}\colon\Dp[Z_1](X_1)\to\Dp[Z_2](X_2)$
satisfying $\Cc$ there exist $\ke\in\D_{Z_1\times Z_2}(\Qcoh(X_1\times
X_2))$ and an isomorphism of exact functors $\fun{F}\iso\FMS{\ke}$.
\end{cor}

\subsection{The category generated by a spherical object}\label{subsec:spherical}

Let us start with a detour about the derived category of a smooth quasi-projective curve with support condition on a closed point for which we can prove variants of Theorems \ref{thm:main2} and \ref{thm:main1}. Notice that in this case it is not true that the
maximal $0$-dimensional torsion subsheaf of $\ko_Z$ is trivial. In
particular, it is easy to see that the construction in Section
\ref{subsec:almostample} does not provide an almost ample set. Thus we
need a particular treatment that, unfortunately, works only for points
embedded in curves.

\begin{prop}\label{prop:spherical}
Let $\p$ be a closed point in a smooth quasi-projective curve $C$.

{\rm (i)}  Let $X$ be a scheme with a closed subscheme
$Z$ and let
\[
\fun{F}\colon\Db_{\p}(C)\lto\Dp[Z](X)
\]
be an exact functor such that
\begin{equation}\label{eqn:FMp}
\Hom_{\Dp[Z](X)}(\fun{F}(\ka),\fun{F}(\kb)[k])=0,
\end{equation}
for all $\ka,\kb\in\Coh_{\p}(C)$ and all $k<0$. Then there exist
$\ke\in\Db_{\{\p\}\times Z}(\Qcoh(C\times X))$ and an isomorphism of
exact functors $\fun{F}\iso\FMS{\ke}$.

{\rm (ii)} The triangulated category $\Db_\p(C)$ has a strongly unique enhancement.
\end{prop}

\begin{proof}
By Proposition \ref{prop:amp} the subcategory $\cat{C}$ of $\Coh_{\p}(C)$, whose objects are
$\{\ko_{n\p}:n>0\}$, satisfies property \eqref{ample2} in Definition
\ref{def:almostample}. In particular, looking carefully at the
construction in Section \ref{subsec:kernel1}, this together with
\eqref{eqn:FMp} is enough to provide an $\ke\in\Db_{\{\p\}\times
  Z_2}(\Qcoh(C\times X))$ and an isomorphism
\[
\theta\colon\fun{F}\rest{\cat{C}}\isomor\FMS{\ke}\rest{\cat{C}}.
\]
To be precise, the fact that $\ke$ is a bounded complex is a
consequence of Lemma \ref{lem:boundedcohom} below.

Set $\cat{D}_0$ to be the (strictly) full subcategory of $\Db_\p(X)$
whose objects are isomorphic to shifts of objects of
$\Coh_{\p}(X)$. By Corollary \ref{cor:critiso1}, the isomorphism
$\theta$ extends (uniquely) to an isomorphism compatible with shifts
\[
\theta_0\colon\fun{F}\rest{\cat{D}_0}\isomor\FMS{\ke}\rest{\cat{D}_0}.
\]
Being $C$ a smooth curve, any object $\kf\in\Db_\p(C)$ can be written
(in an essentially unique way) as a finite direct sum of objects of
$\cat{D}_0$. Thus $\theta_0$ extends (uniquely) to the desired
isomorphism $\fun{F}\isomor\FMS{\ke}$, proving (i).

As for (ii), the same line of reasoning as in the proof of Theorem \ref{thm:main1} in Section \ref{subsec:enh} works with the only difference that the extension of the isomorphism
\eqref{eqn:isoimp} takes place due to Corollary \ref{cor:critiso1}
instead of Corollary \ref{cor:critiso2}. Then we argue as in the
last part of the proof of (i), i.e.\ using that any object in $\Db_\p(C)$ can be (uniquely) written
as a finite direct sum of objects of $\cat{D}_0$.
\end{proof}

Obviously, if $C$ is a smooth quasi-projective curve, the Serre functor of $\Db_\p(C)$ is the shift by one and so $\Db_\p(C)$ is
a $1$-Calabi--Yau category. Moreover, it is easy to verify that the skyscraper sheaf $\ko_\p\in\Coh_\p(C)$ is a $1$-spherical object when $\p$ is a $\K$-rational point. Recall that, as we
mentioned in the introduction, an object $S$ in a triangulated category $\cat{T}$ is \emph{$d$-spherical} (where $d$ is a positive integer) if $\Hom_\cat{T}(S,S[i])$
is trivial if $i\neq 0,d$, while it is isomorphic to $\K$
otherwise.

\begin{remark}\label{rmk:equiv}
Let $\cat{T}$ be an idempotent complete algebraic triangulated category (see \cite{K} for its definition) and let $S\in\cat{T}$ be a $d$-spherical object. In this case, up to equivalence, the full triangulated subcategory $\cat{T}_S$ of $\cat{T}$ classically generated by $S$ does not depend neither on the triangulated category $\cat{T}$ nor on the $d$-spherical object $S$ (see \cite[Thm.\ 2.1]{KYZ}).
\end{remark}

To prove Proposition \ref{prop:sph}, take $S$ to be a $1$-spherical object in an an idempotent complete algebraic triangulated category $\cat{T}$. If $C=\mathbb{A}^1$ and $\p$ is the origin, we observed before that $\ko_\p$ is a $1$-spherical object and hence, by Remark \ref{rmk:equiv}, there exists an exact equivalence $\cat{T}_S\iso\Db_\p(C)$. Now it is enough to apply Proposition \ref{prop:spherical}.

\section{Uniqueness of Fourier--Mukai kernels}\label{sect:uniqueness}

For functors satisfying $\Cc$ and with $X_1$ and $X_2$ smooth, the
uniqueness of Fourier--Mukai kernels is proved via a direct
computation in Section \ref{subsect:supportunique}. As a preliminary
step, we study some basic properties of Fourier--Mukai functors in the
supported setting. In particular, in Lemma \ref{lem:boundedcohom} we show
that the Fourier--Mukai kernels have bounded cohomology. We also make
clear that, in general, one cannot expect the kernel to have coherent
cohomology.

\subsection{Basic properties}\label{subsec:kernel2}

Let $X_1$ and $X_2$ be schemes containing closed subschemes $Z_1$ and $Z_2$. As explained in the following example, we cannot
expect that in general the Fourier--Mukai kernel $\ke$ of a functor
$\FM{\ke}\colon\Dp[Z_1](X_1)\to\Dp[Z_2](X_2)$ has (bounded and)
coherent cohomology when $Z_i\neq X_i$.

\begin{ex}\label{ex:nocoherent}
Suppose that there exists $\ke\in\Db_{Z_1\times Z_1}(X_1\times X_1)$
such that
\[
\FMS{\ke}\iso\id\colon\Dp[Z_1](X_1)\to\Dp[Z_1](X_1).
\]
By \cite[Lemma 7.40]{R}, there exist $n>0$ and $\ke_n\in\Db(nZ_1\times
nZ_1)$ such that $\iota_{1,1}(\ke)\iso(i_n\times i_n)_*(\ke_n)$, where
$i_n\colon nZ_1\to X_1$ is the embedding. For any $\kf_n\in\Db(nZ_1)$,
we have
\begin{equation}\label{eqn:split1}
(i_n)_*(\kf_n)\iso\FM{\iota_{1,1}(\ke)}((i_n)_*(\kf_n))\iso
(i_n)_*\comp\FM{\ke_n}(i_n^*\comp(i_n)_*(\kf_n)).
\end{equation}

Take now $X_1=\PP^k$, $Z_1=\PP^{k-1}$ and $\kf_n:=\ko_{nZ_1}(m)$, for $m\in\ZZ$. An easy calculation shows that $i_n^*\comp(i_n)_*(\kf_n)\iso\ko_{nZ_1}(m)\oplus\ko_{nZ_1}(m-n)[1]$. Hence to have \eqref{eqn:split1} verified, we should have either $\FM{\ke_n}(\ko_{nZ_1}(m))=0$ or $\FM{\ke_n}(\ko_{nZ_1}(m-n))=0$. But the following isomorphisms should hold at the same time
\[
\begin{split}
&\FM{\ke_n}(\ko_{nZ_1}(m))\oplus\FM{\ke_n}(\ko_{nZ_1}(m-n))[1]\iso\ko_{nZ_1}(m),\\
&\FM{\ke_n}(\ko_{nZ_1}(m+n))\oplus\FM{\ke_n}(\ko_{nZ_1}(m))[1]\iso\ko_{nZ_1}(m+n).
\end{split}
\]
If $\FM{\ke_n}(\ko_{nZ_1}(m-n))=0$, then from the second one we would have that $\ko_{nZ_1}(m)[1]$ is a direct summand of $\ko_{nZ_1}(m+n)$ which is absurd. Thus $\FM{\ke_n}(\ko_{nZ_1}(m))=0$. As this holds for all $m\in\ZZ$, we get a contradiction.
\end{ex}

On the other hand, it is easy to find a kernel of the identity
functor. Indeed, denoting by $\Delta\colon X_1\to X_1\times X_1$ the
diagonal embedding, setting $\ko_{\Delta}:=\Delta_*(\ko_{X_1})$, and
defining
\begin{equation}\label{eqn:kerid}
\ki:=\iota_{1,1}^!(\ko_{\Delta})\in\Db_{Z_1\times Z_1}(\Qcoh(X_1\times X_1))
\end{equation}
(notice that $\ki$ has bounded cohomologies by Lemma \ref{lem:sums}),
we have the following result.

\begin{lem}\label{lem:identity}
There exists an isomorphism of exact functors
\[
\id\iso\FMS{\ki}\colon\Db_{Z_1}(\Qcoh(X_1))\to\Db_{Z_1}(\Qcoh(X_1)).
\]
\end{lem}

\begin{proof}
Indeed, $\FMS{\ki}\iso\iota_1^!\comp\FM{\ko_{\Delta}}\comp\iota_1$ by
Lemma \ref{lem:FMS}. This is enough to conclude, since
$\FM{\ko_{\Delta}}\iso\id$.
\end{proof}

The following result proves, in particular, that in Theorem
\ref{thm:main2} the kernel $\ke$ is actually in $\Db_{Z_1\times
  Z_2}(\Qcoh(X_1\times X_2))$.

\begin{lem}\label{lem:boundedcohom}
If $\ke\in\D_{Z_1\times Z_2}(\Qcoh(X_1\times X_2))$ is such that $\FMS{\ke}(\Dp[Z_1](X_1))\subseteq\Dp[Z_2](X_2)$, then $\ke\in\Db_{Z_1\times Z_2}(\Qcoh(X_1\times X_2))$. Moreover, if $X_1=Z_1$, then $\ke\in\Db_{X_1\times Z_2}(X_1\times X_2)$.
\end{lem}

\begin{proof}
By \cite[Thm.\ 6.8]{R}, for $i=1,2$, the category $\D_{Z_i}(\Qcoh(X_i))$ has a compact generator $G_i\in\Dp[Z_i](X_i)$ (see \cite{BB} for the case without support conditions). Moreover, by the explicit description of the compact generator in the proof of \cite[Thm.\ 6.8]{R}, one sees that $G_1\boxtimes G_2$ is a compact generator of $\D_{Z_1\times Z_2}(\Qcoh(X_1\times X_2))$ (for the non-supported case, see for example \cite[Lemma 3.4.1]{BB} and \cite[Thm.\ 3.7]{LP}).

By \cite[Prop.\ 6.9]{R}, the kernel $\ke$ has bounded cohomology if and only if there exists an interval $[a,b]\subset\RR$ such that $\Hom(G_1\boxtimes G_2,\ke[k])=0$, for any $k\not\in[a,b]$. But now
\[
\Hom(G_1\boxtimes G_2,\ke[k])\iso\Hom(\iota_1(G_1)\boxtimes\iota_2(G_2),\iota_{1,2}(\ke)[k])\iso\Hom(G_2,\FMS{\ke}(G_1\dual)[k])
\]
which is non-trivial only for finitely many $k\in\ZZ$.

Suppose that $Z_1=X_1$. Then $\ke\in\Db_{X_1\times Z_2}(X_1\times X_2)$
if and only if $\widetilde{\ke}:=\iota_{1,2}(\ke)\in\Db(X_1\times X_2)$. Since
$\iota_2\comp\FMS{\ke}\iso\FM{\widetilde{\ke}}\comp\iota_1\iso\FM{\widetilde{\ke}}$, the
functor $\FM{\widetilde{\ke}}$ sends perfect complexes to perfect
complexes. Hence we can assume, without
loss of generality, that $Z_i=X_i$, for $i=1,2$. Then it follows from
\cite[Cor.\ 9.13 (4)]{LO}, where the assumption that the functor is fully faithful is not used (see also \cite[Lemma 4.1]{CS1}), that
$\widetilde{\ke}\in\Db(X_1\times X_2)$.
\end{proof}

We recall that a functor
\[
\fun{F}\colon\Db_{Z_1}(\Qcoh(X_1))\lto\Db_{Z_2}(\Qcoh(X_2))
\]
is \emph{bounded} if there is an interval $[a,b]\subset\RR$ such that,
for any $\ka\in\Qcoh_{Z_1}(X_1)$, we have that if
$H^i(\fun{F}(\ka))\neq 0$, then $i\in[a,b]$. We will need the
following lemma.

\begin{lem}\label{lem:bound}
Assume that the base field $\K$ is perfect. Then every exact functor $\fun{F}\colon\Db_{Z_1}(X_1)\to\Db_{Z_2}(\Qcoh(X_2))$ or $\fun{G}\colon\Db_{Z_1}(\Qcoh(X_1))\to\Db_{Z_2}(\Qcoh(X_2))$, commuting with arbitrary direct sums, is bounded.
\end{lem}

\begin{proof}
To deal with the first part, observe that, by Proposition \ref{prop:Ball1}, the category $\Coh_{Z_1}(X_1)$ is generated, as an abelian category, by the image of the natural fully faithful functor $i_*\colon\Coh(Z_1)\mono\Coh_{Z_1}(X_1)$. This means that it is enough to show the boundedness of the functor
\[
\fun{F}':=\fun{F}\comp i_*\colon\Db(Z_1)\lto\Db_{Z_2}(\Qcoh(X_2)).
\]
Now this is a straightforward consequence of \cite[Thm.\ 7.39]{R} (here we need that $\K$ is perfect). Indeed, by this result, there exists a positive integer $d$ and a compact object $G$ in $\D(\Qcoh(Z_1))$ such that any object of $\Db(Z_1)$ is generated by $G$ in at most $d$ steps by taking cones, shifts, direct summands and finite direct sums. Hence $\fun{F}'$ is certainly bounded.

The second part, concerning exact functors between the bounded derived categories of quasi-coherent sheaves, is proved using the same argument. Indeed, one applies again Proposition \ref{prop:Ball1} and reduces to studying the boundness of the exact functor
\[
\fun{G}'\colon\Db(\Qcoh(Z_1))\lto\Db_{Z_2}(\Qcoh(X_2)).
\]
By \cite[Thm.\ 7.39]{R}, there exists a positive integer $d$ such that any object of $\Db(\Qcoh(Z_1))$ is generated by $G$ in at most $d$ steps by taking cones, shifts, direct summands, finite direct sums and arbitrary multiples of the same object (see \cite[Sect.\ 3.1.1]{R}). As $\fun{G}'$ commutes with arbitrary direct sums, this is enough to conclude.
\end{proof}

\subsection{The uniqueness of the Fourier--Mukai kernels}\label{subsect:supportunique}

Assume that the base field $\K$ is perfect and that $X_1$ and $X_2$
are smooth quasi-projective schemes containing projective subschemes
$Z_1$ and $Z_2$.  Consider $\ke\in\Db_{Z_1\times Z_2}(\Qcoh(X_1\times
X_2))$ and observe that, obviously,
$\ke\iso\iota_{1,2}^!(\widetilde{\ke})$, where
$\widetilde{\ke}:=\iota_{1,2}(\ke)\in\Db(\Qcoh(X_1\times X_2))$. We
get a Fourier--Mukai functor
\begin{equation}\label{eqn:Phi_i}
\Psi_{\ke}:=\FM{\ko_{\Delta}\boxtimes\widetilde{\ke}}\colon\D(\Qcoh(X_1\times X_1))\lto\D(\Qcoh(X_1\times X_2)).
\end{equation}
Let $\widetilde{\ki}:=\iota_{1,1}(\ki)\in\Db(\Qcoh(X_1\times X_1))$,
where $\ki$ is the complex defined in \eqref{eqn:kerid} (such that
$\FMS{\ki}\iso\id$ by Lemma \ref{lem:identity}). We first prove the
following result.

\begin{lem}\label{lem:unique1}
We have $\Psi_{\ke}(\widetilde{\ki})\iso\widetilde{\ke}$ and
$\Psi_{\ke}(\kf\boxtimes\kg)\iso
\kf\boxtimes\FM{\widetilde{\ke}}(\kg)$ for every
$\kf,\kg\in\D(\Qcoh(X_1))$.
\end{lem}

\begin{proof}
A standard computation shows that, for every $\kd\in
\D(\Qcoh(X_1\times X_1))$,
\[
\Psi_{\ke}(\kd)\iso\kd\kcomp\widetilde{\ke}:=
(p_{1,3})_*(p_{1,2}^*(\kd)\otimes p_{2,3}^*(\widetilde{\ke})),
\]
where $p_{i,j}$ denote the obvious projections from $X_1\times
X_1\times X_2$. The second assertion in the statement is then
clear. As for the first one, by Lemma \ref{lem:compat} we have
\[
\widetilde{\ki}=\iota_{1,1}\comp\iota_{1,1}^!\comp\Delta_*(\ko_{X_1})
\iso\Delta_*\comp\iota_1\comp\iota_1^!(\ko_{X_1}).
\]
Moreover, it is easy to see that
$\Delta_*(\kf)\kcomp\widetilde{\ke}\iso
p_1^*(\kf)\otimes\widetilde{\ke}$ for every $\kf\in\D(\Qcoh(X_1))$.
It follows that, using the same notation as in the proof of Lemma
\ref{lem:FMS},
\[
\Psi_{\ke}(\widetilde{\ki})\iso
p_1^*\comp\iota_1\comp\iota_1^!(\ko_{X_1})\otimes\widetilde{\ke}\iso
\bar{\iota}_1\comp\bar{\iota}_1^!(\ko_{X_1\times X_2})\otimes
\widetilde{\ke}\iso\bar{\iota}_1\comp\bar{\iota}_1^!(\widetilde{\ke})
\iso\widetilde{\ke},
\]
again by Lemma \ref{lem:compat}.
\end{proof}

We can use this to prove the following result which is precisely the uniqueness statement in Theorem \ref{thm:main2}.

\begin{prop}\label{prop:uniq}
Let $X_1$, $X_2$, $Z_1$ and $Z_2$ be as above and let
$\fun{F}\colon\Db_{Z_1}(X_1)\to\Db_{Z_2}(X_2)$ be an exact functor satisfying $\Cc$. Assume that there are $\ke_1,\ke_2\in\Db_{Z_1\times Z_2}(\Qcoh(X_1\times X_2))$ such that
\[
\fun{F}\iso\FMS{\ke_1}\iso\FMS{\ke_2}.
\]
Then $\ke_1\iso\ke_2$.	
\end{prop}

\begin{proof}
Since $\widetilde{\ki}$ is bounded, there exists a bounded above
complex
\[
\cp{\kl}:=\{\cdots\to\kl^j\to\kl^{j+1}\to\cdots\to\kl^n\to0\}\in\D(\Qcoh(X_1\times X_1))
\]
such that $\cp{\kl}\iso\widetilde{\ki}$ and, for any $j\in\ZZ$, the sheaf $\kl^j$ is of the form $P_j\boxtimes M_j$, where $P_j$ and $M_j$ are (possibly infinite) direct sums of sheaves in $\cat{Amp}(Z_1,X_1,H_1)$ (use Proposition \ref{prop:amp}).

Using again that $\widetilde{\ki}$ is a bounded complex, for $m>0$
sufficiently large, the stupid truncation in position $-m$
\[
\cp{\km}:=\{0\to\kl^{-m}\to\cdots\to\kl^n\to0\}
\]
of $\cp{\kl}$ is such that $\cp{\km}\iso\widetilde{\ki}\oplus\kk[m]$, for some $\kk\in\Qcoh(X_1\times X_1)$. Applying term by term the functor $\Psi_{\ke_i}$ in \eqref{eqn:Phi_i} to $\cp{\km}$ we get a complex of complexes
\[
\Psi_{\ke_i}(\kl^{-m})\lto\cdots\lto\Psi_{\ke_i}(\kl^n).
\]
Due to Lemma \ref{lem:unique1}, the choice of $m$ sufficiently large, the assumption $\Cc$ and Lemma \ref{lem:conv1}, this complex has a unique (up to isomorphism) right convolution
\[
\ka_i:=\widetilde{\ke}_i\oplus\kk_i[m],
\]
with $\kk_i=\Psi_{\ke_i}(\kk)\in\Db(\Qcoh(X_1\times X_2))$.

Applying Lemma \ref{lem:conv2} under the hypothesis $\Cc$, we get
$\ka_1\iso\ka_2$. By Lemma \ref{lem:bound} (here we use that $\K$ is
perfect), the functor $\Psi_{\ke_i}$ is bounded and so, for $m$ large
enough,
\[
\Hom(\widetilde{\ke}_1,\kk_2[m])\iso\Hom(\kk_1[m],\widetilde{\ke}_2)\iso0.
\]
Hence
$\iota_{1,2}(\ke_1)=\widetilde{\ke}_1\iso\widetilde{\ke}_2=\iota_{1,2}(\ke_2)$,
which is equivalent to $\ke_1\iso\ke_2$.
\end{proof}

\begin{remark}\label{rmk:orlov}
Following a suggestion of D.\ Orlov, we can show that if $X_1$ is a projective scheme such that $T_0(\ko_{X_1})=0$, $X_2$ is a scheme and
$\FM{\ke}\colon\Dp(X_1)\lto\Db(X_2)$ is an exact fully faithful functor, then $\ke\in\D(\Qcoh(X_1\times X_2))$ is uniquely determined (up to isomorphism).

Indeed, suppose that there exist $\kf\in\D(\Qcoh(X_1\times X_2))$ and an isomorphism $\FM{\kf}\iso\FM{\ke}$. Consider the quasi-functors
\[
\FMdg{\ke},\FMdg{\kf}\colon\Perf(X_1)\lto\Dg(X_2)
\]
corresponding to $\ke$ and $\kf$ under the bijection given by \cite[Thm.\ 8.9]{To}.
Let $\cat{C}\subseteq\Dg(\Qcoh(X_2))$ be the full dg subcategory whose objects are the same as those in the essential image of $\FM{\ke}$. If $\Psi^{\dg}\colon\cat{C}\to\Perf(X_1)$ is a quasi-functor which is the inverse of $\FMdg{\kf}$ in $\Hqe$, consider the composition
\[
\dgfun{F}:=\Psi^{\dg}\comp\FMdg{\ke}\colon\Perf(X_1)\lto\Perf(X_1)
\]
which has the property $\FM{\ko_{\Delta}}\iso\id\iso H^0(\dgfun{F})$. As in the proof of \cite[Cor.\ 9.13]{LO}, the dg quasi-functor $\dgfun{F}$ extends to $\dgfun{G}\colon\Dg(X_1)\lto\Dg(X_1)$.
On the other hand, by \cite[Thm.\ 8.9]{To}, there exists (a unique)
$\kg\in\D(\Qcoh(X_1\times X_1))$ such that
$\dgfun{G}\iso\FMdg{\kg}$. Hence
\[
\FM{\ko_{\Delta}}\iso H^0(\dgfun{F})\iso\FM{\kg}
\]
and, using for example
\cite[Thm.\ 1.2]{CS1}, we get $\kg\iso\ko_{\Delta}$. Therefore $\dgfun{G}\iso\id$ and so $\FMdg{\ke}\iso\FMdg{\kf}$. Applying again \cite[Thm.\ 8.9]{To}, we deduce $\ke\iso\kf$.

Notice that the proof above does not work if the functor $\FM{\ke}$ satisfies $\Cc$ in the introduction but it is not fully faithful. Nevertheless we expect the result to be true in this case as well.
\end{remark}


\bigskip

{\small\noindent {\bf Acknowledgements.} The authors are mostly grateful to the referee(s) for the insightful comments which allowed them to improve the quality of the paper in a substantial way. We  would like to thank Arend Bayer, Daniel Huybrechts, Emanuele Macr\`i, Dmitri Orlov and Pawel Sosna for comments on an early version of this paper. Part of this article was written while P.S.\ was visiting the Department of Mathematics of the University of Utah and the Institut Henri Poincar\'e in Paris whose warm hospitality is gratefully acknowledged. It is a pleasure to thank these institutions and the Istituto Nazionale di Alta Matematica for the financial support.}


\end{document}